\documentclass[english]{amsart}

\usepackage[utf8]{inputenc}
\usepackage[T1]{fontenc}

\usepackage{amsmath}
\usepackage{amsthm}
\usepackage{amsfonts}
\usepackage{amssymb}
\usepackage{bbold}
\usepackage{mathrsfs}
\usepackage{subcaption}
\usepackage{csquotes}
\usepackage{graphicx}
\usepackage{color}
\usepackage{tikz}
\usetikzlibrary{patterns}
\usepackage[
backend=bibtex,
style=alphabetic,
citestyle=alphabetic,
doi=false,
isbn=false,
url=false]{biblatex}
\bibliography{references}

\usepackage{hyperref}
\usepackage{cleveref}
\usepackage{enumerate}

\title{Iterated Minkowski sums, horoballs and north-south dynamics}
\author{Jeremias Epperlein}
\address{Jeremias Epperlein, Universit\"at Passau, Fakult\"at f\"ur Informatik und Mathematik, Innstra{\ss}e 33, 94032 Passau, Germany}
\email{jeremias.epperlein@uni-passau.de}
\author{Tom Meyerovitch}
\address{Tom Meyerovitch,
	Department of Mathematics, Ben Gurion University of the Negev, P.O.B. 653
	Be'er Sheva 8410501, Israel.}
\email{mtom@math.bgu.ac.il}
\subjclass[2000]{37B10, 37B15, 11H06, 20F65}
\keywords{Minkowski sum, Minkowski product, maximum cellular
automaton, amenability, abstract convexity structure}

\newtheorem{thm}{Theorem}[section]
\newtheorem{lem}[thm]{Lemma}
\newtheorem{defn}[thm]{Definition}
\newtheorem{prop}[thm]{Proposition}
\newtheorem{exam}[thm]{Example}
\newtheorem{quest}[thm]{Question}
\newtheorem{cor}[thm]{Corollary}
\newtheorem{rem}[thm]{Remark}

\newcommand{\CM}{\mathscr{M}}

\newcommand{\Zinfty}{\Z_{\pm \infty}}
\newcommand{\indi}{\mathbb{1}}
\newcommand{\stab}{\mathit{stab}}

\DeclareMathOperator{\Hor}{Hor}
\newcommand{\setsep}{:~}

\DeclareMathOperator{\vol}{vol}
\DeclareMathOperator{\conv}{conv}

\DeclareMathOperator{\Cayley}{Cayley}
\DeclareMathOperator{\Ext}{Ext}

\DeclareMathOperator{\Env}{Env}

\DeclareMathOperator{\Evt}{Evt}

\newcommand{\N}{\mathbb{N}}
\newcommand{\R}{\mathbb{R}}
\newcommand{\Z}{\mathbb{Z}}

\newcommand{\PP}{\mathcal{P}}

\newcommand{\CL}{\mathcal{L}}
\newcommand{\eps}{\varepsilon}
\newcommand{\Fin}{\mathit{Fin}}

\begin{document}

\maketitle
\begin{abstract}
Given a finite generating set $A$ for a group $\Gamma$, we study the map $W \mapsto WA$ as a topological dynamical system --  a continuous self-map of the compact metrizable space of subsets of $\Gamma$. 
If the set $A$ generates $\Gamma$ as a semigroup and contains the identity, there are precisely  two fixed points, one of which is attracting. This supports the initial impression that the dynamics of this map is rather trivial. 
 Indeed,  at least when  $\Gamma= \Z^d$ and $A \subseteq \mathbb{Z}^d$ a finite positively  generating set containing the natural invertible extension of the map $W \mapsto W+A$ is always topologically conjugate to the unique ``north-south'' dynamics on the Cantor set. In contrast to this, 
we show that various    natural ``geometric'' properties of the finitely generated group $(\Gamma,A)$ can be recovered from the dynamics of this map, in particular, the growth type and amenability of $\Gamma$. When $\Gamma = \Z^d$, we show that the volume of the convex hull of the generating set $A$ is also an invariant of topological conjugacy. Our study introduces,  utilizes and develops a certain convexity structure on subsets of the group $\Gamma$,  related to a new concept which we call the sheltered hull of a set. We also relate this study to the structure of  horoballs in  finitely generated groups, focusing on the abelian case. 
\end{abstract}

\section{Introduction}

In this paper we study the topological dynamical system associated to
a finitely generated group via the ``Minkowski product''.  We denote
the collection of subsets of a countably infinite group $\Gamma$ by
$\PP(\Gamma)$. This space is naturally identifiable with the space
$\{0,1\}^\Gamma$, thus naturally equipped with a topology 
turning it into a topological Cantor set.

Any finite subset $A \subseteq \Gamma$ defines a continuous map $\varphi_A:\PP(\Gamma) \to \PP(\Gamma)$ given by $\varphi_A(W):=WA$,
where
\[ WA :=  \{w a \setsep w \in W,~a \in A\}.\]
The set $WA$ is often called  \emph{the Minkowski product} of $W$ and $A$.

 Most instances in the literature  deal with the the abelian case, where the group operation is usually denoted by $+$ and the set $W + A := \left\{ w +a \setsep w \in W,~a \in A\right\}$ is called the \emph{Minkowski sum}.

 Our point of view here is to study $(\PP(\Gamma),\varphi_A)$ as a dynamical system. 
%We call a subset $A$ in a group $\Gamma$ \emph{positively
 %   generating} if $\Gamma=\bigcup_{n=1}^\infty A^n$.  This means that
  %$A$ is a generating set for $\Gamma$ considered as a semigroup.
  %The pair $(\Gamma,A)$ is called a \emph{marked group}. 
We ask how  dynamical properties of  $\varphi_A$ related to algebraic properties of $\Gamma$ and more specifically  to geometric properties of the Cayley graph of $\Gamma$ with respect to the set $A$.  
We focus on the case where $A$ is \emph{positively generating}, i.e.
$\Gamma= \bigcup_{n=1}^\infty A^n$ and $1_\Gamma \in A$.
We call such  a pair $(\Gamma,A)$ as above simply a \emph{finitely generated group}.
%, see e.g. \cite{shalomTao2010} for similar terminology.

In this case $\lim_{n \to \infty} \varphi_A^n(W) = \Gamma$ for any
non-empty set $W \subset \Gamma$ so $\Gamma$ and $\emptyset$ are the
only fixed point for $\varphi_A$, and the forward orbit of any set in $\PP(\Gamma)\setminus \{\emptyset\}$ converges to the fixed point $\Gamma$.
Thus, the only ergodic
$\varphi_A$-invariant probability measures on $\PP(\Gamma)$ are the delta
measures $\delta_{\Gamma}$ and $\delta_{\emptyset}$.
% In fact, up to scaling those are the only ergodic conservative $\sigma$-finite invariant measures.  
So
from the ergodic theory viewpoint, the maps $\varphi_A$ are completely
trivial. 
Nevertheless, from the point of view of topological dynamics it turns out that the system $(\PP(\Gamma),\varphi_A)$
encodes non-trivial properties of  the finitely generated group $(\Gamma,A)$.

Call a property $P$ of finitely generated groups \emph{dynamically
  recognizable} (among a family $\mathcal{G}$ of finitely generated
groups) if for any pair of  finitely generated groups	$(\Gamma_1,A_1)$
and $(\Gamma_2,A_2)$ (in the family $\mathcal{G}$) such that
$(\PP(\Gamma_1),\varphi_{A_1}) \cong (\PP(\Gamma_2),\varphi_{A_2})$,
$(\Gamma_1,A_1)$ has the property $P$ if and only if $(\Gamma_2,A_2)$
has the property $P$.
This approach is analogous to the study of group properties only
depending on the Cayley graph in geometric group theory. 

We show that the following properties are dynamically recognizable:
\begin{itemize}
	\item Growth type (polynomial, exponential, \ldots), see \Cref{cor:growth_invariant}.
	\item Rank and volume of the convex hull of the generating set, among $\{\Z^d \setsep d \in \N \}$, see \Cref{cor:vol_conv_invariant}.
	\item The exponential growth rate among e.g. free groups,
     \Cref{cor:exp_growth_inv}. %{\color{blue} hyperbolic groups?}
	\item Amenability, see \Cref{cor:amenablity_recognizable}.
\end{itemize}

From the geometric point of view, the dynamical invariants underlaying
the above results are related to a certain convexity structure we
introduce on subsets of an arbitrary  finitely generated group
$(\Gamma,A)$, and a new concept which we call the \emph{sheltered hull}. This is introduced in \Cref{sec:sheltered}. The convexity structure given by the sheltered hull is related to the notion of a \emph{horoball}, which we recall and discuss in \Cref{sec:horoball}.

The dynamical system $(\PP(\Gamma),\varphi_A)$  only depends on
the directed Cayley graph $\Cayley(\Gamma,A)$.
One can easily generalize the map $\varphi_A$ to
general countable, locally finite, directed graphs 
(see \Cref{sec:horoball}), and for many results we don't actually need
that this graph is a
Cayley graph. Hence this is the level of generality we assume if
it doesn't add further complications.

The dynamically recognizable properties above, captured by the
sheltered hull, are in some sense based on ``quantifying
non-invertibility'' of the map $\varphi_A$ in a manner which is
invariant under topological conjugacy. It is thus natural to ask if there
are dynamically recognizable properties that can be ``detected
disregarding non-invertibility''.  Thus, in
\Cref{sec:eventual_image_natural_extension} we consider ``the
invertible analog'' of $\varphi_A$: Namely, the natural extension of the
restriction of $\varphi_A$ to its eventual image (for brevity we refer
to this as ``the natural extension of $\varphi_A$''). This is a
homeomorphism with ``north-south dynamics'', which we discuss and
recall in \Cref{sec:north_south}.  In \Cref{sec:nat_ext_Z_d} we show that,
at least for $\Gamma=\Z^d$, the natural extensions are all
topologically conjugate to one another.

The first part of the paper, up to
\Cref{sec:eventual_image_natural_extension} deals with results about
general finitely generated groups (in some cases these are general
results about locally finite graphs). In the second part of the paper,
from \Cref{sec:Z_d_geo_horoballs} onwards, we specialize to the case
$\Gamma=\Z^d$. In \Cref{sec:Z_d_geo_horoballs} we discuss the
structure of horoballs in $\Z^d$ with respect to a generating set $A$
and show that up to translation these are in a natural bijection with
the faces of the convex hull of $A$ (viewed as a subset of $\R^d$). En
route we recall some old results about the structure of iterated
Minkowski sums in $\Z^d$. In \Cref{sec:Z_d_top_dynam}, still working
with $\Gamma=\Z^d$, we consider the space of horoballs (and its
closure) as a topological space and a $\Z^d$-dynamical system. In
particular, we observe that its homeomorphism type is uniquely
determined by  the rank $d$ (\Cref{cor:Hor_Z_d_top_invariant}). We
also provide an explicitly  checkable characterisation of the topological conjugacy class of the associated dynamical system (\Cref{prop:conjugacy_of_Hor_ZD_A}). Our proof in  \Cref{sec:nat_ext_Z_d} that the natural extension of $(\PP(\Z^d),\varphi_A)$ is perfect, thus topologically conjugate to the unique north-south dynamics on the Cantor set, is contrasted in \Cref{sec:eventual-image-zd}, where we show that the topological structure of the eventual image is sensitive to the specific generating set $A \subset \Z^d$, already when $d=2$. In \Cref{sec:factoring} we discuss the problem of when $(\PP(\Gamma_1),\varphi_{A_1})$ factors onto  $(\PP(\Gamma_2),\varphi_{A_2})$, a problem which for the most part we have not been able to resolve.

\subsection*{Acknowledgements}
We thank Ville Salo for valuable comments on a preliminarily version of this work, and for sharing with us examples of non-Busemann horoballs.  
Jeremias Epperlein thanks Ben-Gurion University of the Negev, where this work
was written, for its hospitality and stimulating atmosphere. Tom Meyerovitch thanks the University of British Columbia and the Pacific Institute of Mathematical  Sciences for an  excellent Sabbatical, partly overlapping this work. This
research was supported by a post-doctoral research fellowship from
the Minerva Foundation and by the Israel Science Foundation (grant number 1052/18).

\section{Horoballs in directed graphs and finitely generated groups}\label{sec:horoball}

\subsection{Quasi-metrics on directed graphs}

Let $G=(V(G),E(G))$ be a countably infinite,  locally finite
directed graph. We assume throughout that $G$ is strongly connected, meaning that for every $v,w \in V(G)$ there is a directed path in $G$ from $v$ to $w$.

Denote the collection of subsets of $V(G)$ by $\PP(V(G))$,
which we identify with $\{0,1\}^{V(G)}$, equipped with the product
topology.  
Define a continuous self-map
$\varphi_G:\PP(V(G)) \to \PP(V(G))$ by
\[\varphi_G(W) := W \cup \{v \in V(G):~ \exists w \in W \mbox{ s.t. } (w,v) \in E(G)\}.\]
Via the identification of subsets of $V(G)$ with their characteristic
functions this can be rewritten as follows:
$$\indi_{\varphi_G(W)}(w) = \max \{ \indi_W(v) \setsep (w,v) \in E(G) \mbox{ or } w=v\}.$$

Under the assumption that  the graph $G$ is strongly connected, there are precisely two
fixed points for $\varphi_G$, namely the $V(G)$ and $\emptyset$. Also,
for any $W \in \PP(V(G)) \setminus \{\emptyset\}$,
$\lim_{n \to \infty} \varphi_G^n(W)= V(G)$. For $v,w \in V(G)$ define
\[d(w,v)=\min \{n \in \N_0 \setsep v \in \varphi_G^n(\{w\})\}.\] The
function $d:V(G) \times V(G) \to \N_0$ is a \emph{quasi-metric},
meaning it satisfies the axioms of a metric, apart from symmetry.  For
$w,v \in V(G)$ $d(w,v)$ is the minimal number of edges in a
directed path from $w$ to $v$.  For background on quasi-metrics,
sometimes also called ``non-symmetric metrics'', see for instance
\cite{kellyBitopological1963,wilsonQuasiMetric1931}. In the case where
$G$ is an undirected graph (which we think of as a directed graph with
edges going both ways), $d$ is called the graph metric. The reason we
allow for directed graphs is to handle non-symmetric
generating sets in finitely generated groups.  We can also express
$\varphi_G^n$ directly in terms of the quasi-metric as
$\varphi_G^n(W)=\{v \in V(G) \setsep \exists w \in W: d(w,v)\leq
n\}$. Many notions in metric geometry still make sense in our setting:

We define a \emph{geodesic ray} in a graph $G$ as
a sequence of vertices
$(\gamma_{n})_{n \in \N_{0}}$ in $V(G)$ such that
the shortest directed path from $\gamma_{i}$ to $\gamma_j$ for $i>j$
has length $i-j$. In particular, this means that
$(\gamma_{n+1},\gamma_n)_{n \in \N_0} \in E(G)$ for every $n$.

For a vertex $v \in V$, we refer to
$\varphi_G^n(\{v\})$ as \emph{the ball of radius $n$ centered at $v$}. In the
case where the graph $G$ is undirected, this is actually a ball with
respect to the graph metric.  In general, for $W \subseteq V(G)$ one can think of
$\varphi_G^n(W)$ is \enquote{the set of elements accessible from $W$ in
$n$ steps}. 

The space of balls is clearly
invariant under $\varphi_G$ but it is not a closed subset of $\PP(V(G))$. The new elements
arising in the closure  are called horoballs.

 \begin{defn}\label{defn:horoball}
Let $G$ be a locally finite graph.
 	A  limit point in $\PP(V(G))$ of
   a sequence of balls $(\varphi_{G}^{n_i}(\{v_i\}))_{i \in \N}$
   with radii $(n_i)_{i\in \N}$ tending to infinity is called a \emph{horoball} in $G$ if it is  non-empty and has
   non-empty complement. Let $\Hor(G)$ be the set of all horoballs
   in $G$.
\end{defn}

The definition of horoballs is due to Gromov
\cite{gromovHyperbolic1978} and comes from hyperbolic geometry.
It makes sense in any quasi-metric space. Gromov's definition
is slightly different from ours, but we will see later in this section
that they are equivalent. For the definition in terms of limits of balls, see for instance
 \cite{meyerovitchSaloPointwisePeriodicity2019}. See also
 \cite{auslanderGlasnerWeissRecurrence2007}, where horoballs ``tangent
 to a base point'' are called ``cones'' (in the context of finitely
 generated groups, where the base point is the identity element of the
 group).
 
\subsection{Busemann balls and Gromov's
 horofunction boundary}
Busemann balls are  horoballs coming from geodesic rays:
 \begin{defn}\label{defn:busemannball}
  A \emph{Busemann ball} in $G$ is a set $H \in
   \PP(V(G))\setminus \{\emptyset, V(G)\}$ of the
   form $\bigcup_{r=0}^\infty \varphi_G^r(\{\gamma_r\})$, where
   $(\gamma_{n})_{n \in \N_{0}} \in V(G)^{\N_{0}}$ is a geodesic ray. Note that this is
   an increasing union, so $H$ is indeed a limit of balls. 
 \end{defn}

 There is a slightly more classical approach to horoballs via Gromov's
 horofunction boundary \cite{gromovHyperbolic1978}, which we now
 recall (restricting to the case where the underlaying quasi-metric
 space is a locally finite directed graph). See
 \cite{walshHorofunction2014} and references therein for background and
 further details.  Fix a base vertex $v_0$ in $V(G)$ and consider the
 space $C(V(G),\Z)$ of integer valued continuous functions on $V(G)$
 with the topology of pointwise convergence (in a more general setting
 one needs to consider real-valued functions, with the topology of
 uniform convergence on compact sets).
%All our functions will be integer valued
%and in this case this coincides with pointwise convergence.
%For $(u,v) \in V(G) \times V(G)$
%let $d(u,v)$ be the length of
%the shortest directed path from $u$ to $v$.
%This map fullfills all the axioms of a metric
%besides symmetry.
We have an embedding $\varrho: V(G) \to C(V(G),\Z)$ via
$w \mapsto \varrho_w$ with $\varrho_w(v)=d(w,v)-d(w,v_0)$.
%Let $\CF$ be the closure of the image $\varrho(V(G))$ in $C(V(G),\Z)$.
The elements of $C(V(G),\Z)$ which are on the closure of
$\varrho(V(G))$ but not in $\varrho(V(G))$ are called
\emph{horofunctions}.  By definition, for every $x \in V(G)$, the
sublevel sets of the function $\varrho_x$ are balls around $x$. As in
\cite{walshHorofunction2014}, horoballs are sometimes defined as
sublevel sets of horofunctions. The following proposition verifies
that in our setting sublevel sets of horofunctions are exactly
accumulation points of balls (one implication follows directly from
\cite[Proposition $2.8$]{walshHorofunction2014}:
%The sublevel sets
%of horofunctions are horoballs and vice versa
%every horoball is obtained as the sublevel set of a horofunction.
\begin{prop}\label{lem:horofunction-sublevel}
	Let $F$  be a horofunction on $G$. For every $r \in \Z$ the
	sublevel set $H_{F,R}:=\{v \in V(G) \setsep F(v)\leq R\}$
	is either a horoball or equal to $V(G)$. 
	Conversely, if $H$ is a horoball, then there is 
	$R \in \Z$ and a horofunction $F$ such that
	$H=H_{F,R}$.
\end{prop}
\begin{proof}
Let $(\varrho_{w_k})_{k \in \N}$ be a sequence of functions as above converging
to a horofunction $F$. Then
\begin{align*}
  B_k &:= \{v \in V(G) \setsep
  \varrho_{w_k}(v)\leq R\} = \{v \in V(G) \setsep d(w_k,v) \leq
          d(w_k,v_0) + R\} \\
  &\phantom{:}= \varphi_G^{d(w_k,v_0)+R}(\{w_k\})
\end{align*}
is a sequence of balls converging to $H_{F,R}$.

The set $H_{F,R}$ is non-empty because every $B_k$  has a non-empty
intersection with $\{v_0\} \cup \{v \in V(G)
\setsep d(v,v_0)=-R\} $. Thus, $H_{F,R}$ is either a horoball or equal to $V(G)$.

%TODO: Should the whole set and the empty set be treated as horoballs.

Conversely let $(\varphi_G^{r_k}(\{w_k\}))_{k \in \N}$ be a sequence of balls converging to the horoball $H$.
We will show that $r_k-d(w_k,v_0)$ is bounded.
Let $u$ be a point on the \enquote{boundary of $H$}, namely $u \in \varphi_G(H) \setminus H$.
For sufficiently large $k$ we have $d(w_k,u)=r_k+1$
and hence (keeping in mind that $d$ is not
necessarily symmetric),
\[|r_k-d(w_k,v_0)| = |1+d(w_k,u)-d(w_k,v_0)| \leq 1+d(v_0,u) + d(u,v_0).\]
Therefore we can find a subsequence of balls
$\varphi_G^{r_k}(\{w_k\})$
converging to $H$ with $r_k-d(w_k,v_0)$
constant equal to $R \in \N_0$.

% d(a,b)+d(b,c)>=d(a,c)

Next we will show that the corresponding functions $\varrho_{w_k}$ are
uniformly bounded. Namely for $v \in V(G)$ we have
$|\varrho_{w_k}(v)|=|d(w_k,v)-d(w_k,v_0)|\leq d(v,v_0) + d(v_0,v)$ by the
triangle inequality. Therefore we can select a subsequence of
$(w_k)_{k \in \N}$ such that $\varrho_{w_k}$ converges to a
horofunction $F$. But then $H_{F,R}$ is the limit of the
sequence of sets
\[
	\{v \in V(G) \setsep \varrho_{w_k}(v) \leq R=r_k-d(w_k,v_0)\}
	= 
	\{v \in V(G) \setsep d(w_k,v) \leq r_k\}
	=
	\varphi_G^{r_k}(\{w_k\})
\] and hence $H_{F,R}=H$.	
\end{proof}
Thus, for fixed $R \in \Z$,  the map $F \mapsto H_{F,R}$ from the
horofunction boundary to the space of horoballs (and possibly the
whole vertex set) is surjective. %Since way more results are known about horofunctions then horoballs,
It is natural to ask the following.
\begin{quest}
  When is the map $F \mapsto H_{F,R}$ from the horofunction boundary
  to the space of horoballs (and possibly the whole vertex set)
  injective for fixed $R \in \Z$?
\end{quest}

% \begin{exam}
%    If $H$ is a horoball, a natural candidate for a horofunction
%    with sublevel set $H$ is $\dist(X\setminus H, \cdot)-1$.
%    However, this function is not necessarily the only horofunction
%    generating the horoball
%    (maybe it is not a horofunction at all sometimes?)
%    as the following example shows.
% 	Let $\Gamma$ be a group  with infinitely many dead ends with respect to
% 	the generating set $A$, for example the discrete Heisenberg group.
% 	Now let $x_i$ be a sequence of pairwise different dead ends of depth $2$
% 	and let $r_i:=d(x_i,e)$. Consider the balls $\varphi_A^{r_i}(\{x_i^{-1}\})$.
% 	The identity element $e$ is contained in these balls, and the distance from $e$ to $\Gamma \setminus \varphi_A^{r_i}(\{x_i^{-1}\})$
% 	is $2$. Set $x_0:=e$
% 	we have $\varrho_{x_i}(e)=0$. Select a subsequence such that $\varrho_{x_i}$ converges
% 	to a horofunction $F$. Now the sublevel set $H:=\{y \in \Gamma \setsep F(e)\leq 0\}$
% 	is a horoball, but $F(e)=0\neq \dist(e,\Gamma \setminus H)-1=1$.
% \end{exam}

If we take the limit along a geodesic ray $\gamma$ starting in
$x_0 \in V(G)$, the functions $\varrho_{\gamma_k}$ converge to a
special kind of horofunction, called a \emph{Busemann function}, see
\cite{rieffelGroupAlgebrasCompact2002, walshActionNilpotentGroup2011}.
The sublevel sets of Busemann functions are precisely the Busemann
balls.  There are known examples of Cayley graphs with horofunctions
which are not Busemann functions
\cite{walshBusemann2009,walshNilpotent2011,,websterWinchester2006}.
In response to a question raised in a preliminary version of this
work, Salo \cite{saloExamples2020} provided various examples of Cayley
graphs with non-Busemann horoballs and showed that connectedness of every
horoball is equivalent to ``almost-convexity'' as introduced by Cannon
\cite{cannon1987almostConvex}. Furthermore, Salo shows that the
lamplighter group has horoballs which are not even coarsely connected.
%We pose the following related question:
%\begin{quest}\label{que:non-Busemannball}
% Is there a Cayley graph for  a finitely generated group that contains a horoball which is not a
%  Busemann ball?
%\end{quest}
%
%A negative answer to \Cref{que:non-Busemannball} would imply for
%instance that within the class of Cayley graphs every horoball is an
%increasing union of balls and that every  horoball is the limit of a sequence of balls the form
%  $(\varphi_A^n(\{v_n\}))_{n \in \N}$ (without passing to a subsequence).  We do not know if these properties  hold for a general  horoball. 
%While we do not know the answer to \Cref{que:non-Busemannball}, the following observations  are partial substitutes:
The following observation shows that horoballs that are minimal with respect to inclusion are Busemann:

\begin{prop}\label{thm:Busemannball}
	Let $H$ be a horoball and $v \in H$. There is a Busemann ball
	$B$ such that $v \in B \subseteq H$.
\end{prop}
\begin{proof}
  Let $\varphi_G^{r_k}(\{w_k\})$ be a sequence of balls converging to
  $H$.  For each $k$, let 
  $\gamma^k= (\gamma^k_{\ell_k},\gamma^k_{\ell_k-1}, \dots,\gamma^k_{0})$ be a shortest path
  in $G$
  starting in $w_k=\gamma^k_{\ell_k}$ and ending in $v =\gamma^k_0$. 
    Because $G$ is locally finite, and in particular has finite
    in-degrees, we can assume that for every $\ell \in \N_{0}$
    the sequence
    $(\gamma^{k}_{\ell})_{k \in \N_{0}}$ stabilizes, hence
    we can assume that the sequences $\gamma^k$
  converge to a geodesic ray $\gamma= (\gamma_k)_{k \in \N_{0}}$ with
  $\gamma_0=v$. The Busemann ball corresponding to the geodesic ray
  $\gamma$ is contained in $H$ and contains $v$.
\end{proof}

Since unions of horoballs will play the role of the eventual image for the dynamical systems $(\PP(V(G)),\varphi_G)$,
we mention the following conclusion of \Cref{thm:Busemannball}:
\begin{cor}
Every set which is a union of horoballs is also a union of Busemann balls.
\end{cor}

As we will see in \Cref{sec:Z_d_geo_horoballs}, in every Cayley graph of  $\Z^n$  there are 
only finitely many horoballs up to translation.  It is thus natural to ask:
\begin{quest}
	Which Cayley graphs allow for finitely many horoballs up to translation?
%	for a fixed generating set? Can this depend on the generating set?
 \end{quest}

Tointon and Yadin \cite[Conjecture $1.3$]{tointonYadin2016} ask if
groups of polynomial growth admit finitely many horofunctions, and
recall an observation  of Karlsson that an affirmative solution would
yield an alternate proof of Gromov's theorem on groups of polynomial growth.

\subsection{Iterated Minkowski products, positively generating sets and Cayley graphs} 
Recall that a subset $A$ in a group $\Gamma$ is called \emph{positively
    generating} if $\Gamma=\bigcup_{n=1}^\infty A^n$.  This means that
  $A$ is a generating set for $\Gamma$ considered as a semigroup.
  Throughout the paper when we consider a finitely generated group $(\Gamma,A)$ we will assume $A$ is a finite positively generating set which contains the identity, although we will not necessarily  assume $A$ is symmetric. The Cayley graph $\Cayley(\Gamma,A)$ for the finitely generated group $(\Gamma,A)$ has vertex set $\Gamma$ and edges of the form $(g,ga)$ with $g \in \Gamma$ and $a \in A$.
In this case the map $\varphi_{\Cayley(\Gamma,A)}$ coincides with the Minkowski product maps
$\varphi_A$ given by
\[\varphi_A(W) := W A.\]
Note that $\varphi_A$ is a cellular automaton over the group $\Gamma$,
in the sense that it is a continuous map that commutes with the
$\Gamma$ action of translation from the left.
 \begin{defn}\label{defn:horoball_grp}
For a finitely generated group $(\Gamma,A)$ an
   \emph{$A$-horoball} is a non-empty subset of $\Gamma$ that is the
   limit of balls $\{g_nA^{r_n}\}$ that is neither a ball itself
   nor the whole set $\Gamma$.  We abbreviate $ \Hor(\Cayley(\Gamma,A))$ by $\Hor(\Gamma,A)$.
 \end{defn}

 \begin{exam}\label{exam:horoball-z2}
 	Consider $\Gamma=\Z^2$ with the generating set $A=\{-1,0,1\}^2$.
 	Then horoballs are either translated vertical or horizontal halfspaces
 	or translated quadrants. As we will show, up to translation there are therefore eight of them, corresponding to the $4$ edges and $4$ vertices of the convex hull of $A$, which is a square.
 \end{exam}
 
 We note that
 $\overline{\Hor(\Gamma,A)}=\Hor(\Gamma,A) \cup \{\emptyset, \Gamma\}$
 is a closed $\Gamma$-invariant subset of $\PP(\Gamma)$. It is natural to wonder what properties of the group $\Gamma$ or of the generating set $A$ can be extracted from the topology of $\overline{\Hor(\Gamma,A)}$ or from the dynamics of the $\Gamma$ action on $\overline{\Hor(\Gamma,A)}$. 

\begin{quest} 
	Let  $A_1,A_2$ be two positively generating sets for $\Gamma$. Is it the case that $\overline{\Hor(\Gamma,A_1)}$ and $\overline{\Hor(\Gamma,A_2)}$ are  homeomorphic?
\end{quest}
We will later provide an affirmative answer
 in the particular case $\Gamma =\Z^d$ (see \Cref{cor:Hor_Z_d_top_invariant} below).

\section{The sheltered hull}\label{sec:sheltered}
From now on let $G$ be a locally finite strongly connected
directed graph. Our next goal is to introduce a convexity
structure on the vertices of $G$. We will  use this convexity structure in  later sections  to extract an invariant  of topological conjugacy. This new convexity  structure might also be of independent interest from the point of view of geometric group theory.

%We will mainly be interested in the case that $A$ is a 
%symmetric generating set of $\Gamma$ containing the unit element. 

\begin{defn}
	Given  $W \subseteq V(G)$ and $r>0$ we define the
   \emph{$r$-sheltered hull} of $W$ to be
\begin{equation}\label{eq:r_sheltered_def}
S_G^r(W):=\{v \in V(G) \setsep \varphi_G^r(\{v\}) \subseteq \varphi_G^r(W)\}.
\end{equation}

The \emph{sheltered hull} of $W$ is then given by
\begin{equation}\label{eq:sheltered_def}
S_G(W):=\bigcup_{r=1}^\infty S^r_G(W).
\end{equation}
We call a set $W \subseteq V(G)$ \emph{sheltered} if it agrees with its sheltered hull.
\end{defn}

%we define
%	\begin{align*}
%	\Shel^r_G(W)&:= \varphi_G^{-r}\left(\left\{\varphi_G^r(W)\right\}\right),
%	&\Shel_G(W)&:= \bigcup_{r \in \N} \Shel^r_G(W), \\\\
%	S^r_G(W)&:=\bigcup \Shel^r_G(W),
%	&S_G(W)&:=\bigcup_{r \in \N} S^r_G(W).
%	\end{align*}
%	We call $S_A(W)$ the \emph{sheltered hull} of $W$.
%	The reason for this will become clear soon.
%\end{defn}

%In other words, $W \in \Shel^r_G(W)$ iff $\varphi_G^r(W)=\varphi_G^r(V)$.
%Since $S^r_G(W)$ and $S_G(W)$ will play a central role in the following
%investigations we give a few alternative characterisations.

%\begin{prop}\label{prop:simple-description-s-r-a-w}
%	\begin{align*}
%	S_G^r(W)&=\{v \in V(G) \setsep \varphi_G^r(\{v\}) \subseteq \varphi_G^r(W)\}.
%	\end{align*}
%\end{prop}

\begin{rem}
  In \enquote{mathematical morphology}, a subfield of image
  analysis, the operation $W \mapsto S^r_G(W)$ is called
  \emph{closing}, see e.g. \cite{serraImageAnalysisMathematical1982}.
\end{rem}

\begin{rem}
The definition of the sheltered hull naturally extends to general metric and quasi-metric spaces. 
In Euclidean space, the closure of the sheltered hull of a set agrees with the  closed convex hull. 
\end{rem}

We denote by $\overline{G}$ the graph with the same vertex set as $G$ and
the directions of edges in $G$ reversed, that is,
$E(\overline{G})=\{(w,v) \setsep (v,w) \in E(G)\}$.

There is an interesting characterization of the complement of $S_G^r(W)$ in terms of $r$-balls of $\overline{G}$:
\begin{prop}\label{prop:complement-shel-r-a-w}
	A vertex $v \in V(G)$ is contained in the complement of $S^r_G(W)$ if and only if it
	is covered  by an $r$-ball in  $\overline{G}$ which is disjoint from $W$. More precisely, 
	\begin{align*}
	V(G) \setminus S^r_G(W)&=\bigcup \{ \varphi_{\overline{G}}^r(\{u\}) \setsep 
	u \in V(G), \;
	\varphi_{\overline{G}}^r(\{u\}) \cap W = \emptyset\}.	
	\end{align*}
\end{prop}
\begin{proof}
  Suppose $u \in V(G)$ and
  $\varphi_{\overline{G}}^r(\{u\}) \cap W = \emptyset$ or, in other
  words, $u \not\in \varphi_G^r(W)$.
  Let $v \in \varphi_{\overline{G}}^r(\{u\})$.  Then
  $u \in \varphi_G^r(\{v\})$  and therefore
  $\varphi_G^r(\{v\}) \not\subseteq \varphi_G^r(W)$. Thus
  $v \not\in S^r_G(W)$.
	
	Now let $v \in V(G) \setminus S_G^r(W)$. Then $\varphi_{G}^r(\{v\})$ contains 
	at least one element $u$ not contained in $\varphi_G^r(W)$.
	This implies $\varphi_{\overline{G}}^r(\{u\}) \cap W = \emptyset$.
\end{proof}

In analogy to \Cref{prop:complement-shel-r-a-w} we can characterize the complement of $S_G(W)$
as the union of all horoballs in $\overline{G}$ disjoint from $W$. See \Cref{fig:horoball} for an illustration.

\begin{figure}
	\begin{center}
	\includegraphics[width=0.5\textwidth]{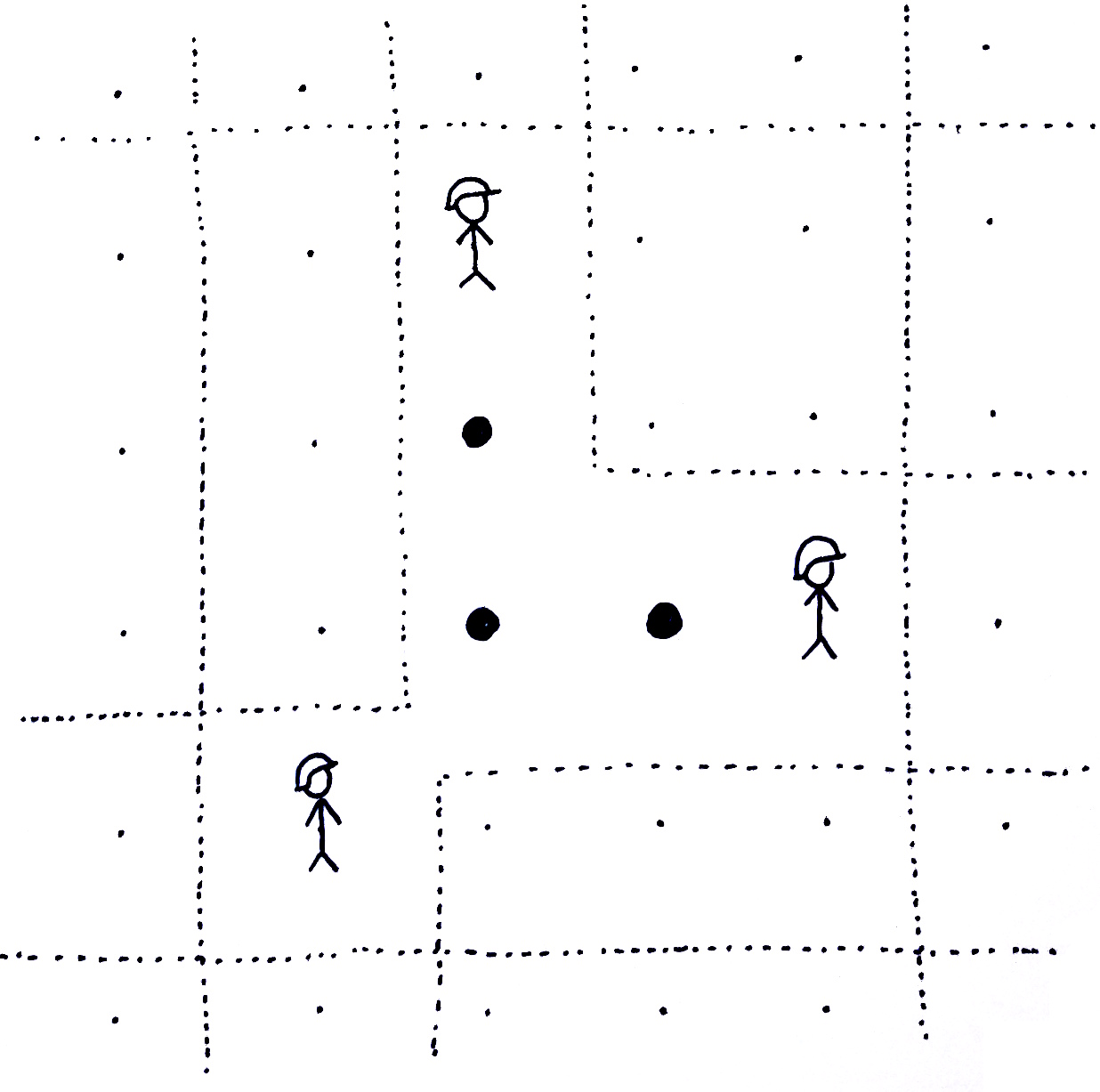}
	\end{center}
	\caption{The sheltered hull of three elements (depicted by little
     guards providing shelter from the evil horoballs)
	of $\Z^2$ with generator $\{-1,0,1\}^2$ and 
	the horoballs covering the complement.}
	\label{fig:horoball}
\end{figure}

\begin{prop}\label{prop:complement-shel-a}
	For every non-empty subset of $W \subseteq \Gamma$ we have
	\[V(G) \setminus S_G(W) = \bigcup\{H \in \Hor(\overline{G}) \setsep H \cap W = \emptyset \}.\]
\end{prop}
\begin{proof}
	Let $v \in V(G) \setminus S_G(W)$. By \Cref{prop:complement-shel-r-a-w} for every $r \in \N$ 
	there is a vertex $u_r \in V(G)$ such that $\varphi_{\overline{G}}^r(\{u_r\})$ is 
	disjoint from $W$ but contains $v$.
	Let $(r_k)_{k \in \N}$ be a growing sequence of radii such that 
	$\varphi_{\overline{G}}^{r_k}(\{u_{r_k}\})$ converges to a limit $H$.
	Since $v$ is contained in $H$ and $H \cap W = \emptyset$ this
	limit $H$ is neither empty nor the whole of $V(G)$. Hence we found
	a horoball $H$ with $v \in H$ and $H \cap W = \emptyset$.
	
	Now let $v$ be in $S_G(W)$. By definition there is $r \in \N$ with
   $\varphi_G^r(\{v\}) \subseteq \varphi_G^r(W)$.  Since
   $\varphi_G^r(\{v\})$ is finite, there is a finite subset
   $\tilde{W}$ of $W$ such that already
   $\varphi_G^r(\{v\})\subseteq \varphi_G^r(\tilde{W})$ and thus
   $v \in S_G^R(\tilde{W})$ for all $R \geq r$.  In particular every
   ball in $\overline{G}$ with radius $R \geq r$ containing $v$ has
   non-trivial intersection with $\tilde{W}$ by
   \Cref{prop:complement-shel-r-a-w}.  Since $\tilde{W}$ is finite
   this also means that every horoball in $\overline{G}$ containing
   $v$ must intersect $\tilde{W} \subseteq W$.  Therefore $v$ is not
   contained in
   $\bigcup\{H \in \Hor(\overline{G}) \setsep H \cap W = \emptyset
   \}$.
\end{proof}

The definition of a sheltered set extends to metric (and quasi-metric) spaces. 
In $\mathbb{R}^d$, with respect to the euclidean metric, the closure of the sheltered hull of a set $A \subseteq \mathbb{R}^d$ coincides with the closure of the convex hull  of $A$.  There is an analogy between sheltered sets in graphs and 
convex subsets of $\mathbb{R}^d$, as  sheltered sets actually form an abstract convexity structure in the  sense of \cite{vandevelTheoryConvexStructures1993}:
\begin{prop}\label{thm:sheltered_conv}
For any directed graph $G$ the  sheltered sets in $G$ fulfill the following axioms:
 \begin{enumerate}[(a)]
 	\item $\emptyset, V(G)$ are both sheltered.
    \item The family of sheltered sets is closed under arbitrary intersections.
    \item The union of an increasing chain of sheltered sets is sheltered.
 \end{enumerate}
Furthermore, for any $W \subseteq V(G)$ the sheltered hull $S_G(W)$ is equal to the intersection of all sheltered sets containing $W$.
\end{prop}
\begin{proof}
	\begin{enumerate}[(a)]
		\item This is clear by definition.
		\item Let $\CM$ be a family of sheltered sets. By
        \Cref{prop:complement-shel-a} we know that a set $W$ is
        sheltered if and only if every element in its complement is
        contained in a horoball in $\overline{G}$ disjoint from
        $W$. For every $v \in V(G) \setminus \bigcap \CM$ there is
        $W \in \CM$ with $v \in V(G) \setminus W$ and hence there is a
        horoball $H$ disjoint from $W$ with $v \in H$ . But then
        $H$ is also disjoint from $\bigcap \CM$ and thus $\bigcap \CM$
        is sheltered.
    \item Let $W_1 \subseteq W_2 \subseteq W_3 \subseteq \dots$ be an
      increasing chain of sheltered sets with $W = \bigcup_{k=1}^\infty W_k$.  Let
      $v \in V(G) \setminus W$ and hence
      $v \in V(G) \setminus W_k$ for all $k \in \N$. We want to find
      a horoball in $\overline{G}$ disjoint from $W$ containing $v$. By
      \Cref{prop:complement-shel-r-a-w} we know that for each
      $k \in \N$ we can find a ball $\varphi_{\overline{G}}^{r_k}(\{v_k\})$ which
      contains $v$ and which is disjoint from $W_k$.  By compactness
      we can take a subsequence of these balls converging to a
      horoball $H$ in $\overline{G}$.  Clearly $v \in H$ and if there
      would be $u \in H \cap W$, then for sufficiently large $k$ we
      would have $u \in \varphi_{\overline{G}}^{r_k}(\{v_k\})$ and $u \in W_k$, since the
      later sets are monotonically increasing.  But the set $W_k$ and
      $\varphi_{\overline{G}}^{r_k}(\{v_k\})$ are disjoint, thus $H \cap W =
      \emptyset$. Therefore $W$ is sheltered.
	\end{enumerate}
\end{proof}

\begin{rem} 
  By \Cref{prop:complement-shel-a} and \Cref{thm:sheltered_conv} the
  sheltered sets form a convexity structure generated by the complements
  of horoballs in the sense that they are the smallest family of sets
  closed under intersections and unions of increasing chains which
  contains horoball complements.  Notice the analogy with the
  situation in $\R^n$ where the usual convexity structure is generated
  by the set of all half spaces (or equivalently the complements half
  spaces). A similar concept of convex hull was considered in
  \cite{fletcherHoroballHullsExtents2011}, but there the convexity
  structure is generated by the horoballs instead of their complements.
 \end{rem}

 \begin{rem}
   Since every horoball is the union of Busemann balls by
   \Cref{thm:Busemannball}, a set $M$ is sheltered if and only if
   every point in its complements is contained in a Busemann ball
   disjoint from $M$. In other words, the convexity structure of
   sheltered sets is also generated by the complements of Busemann balls.
 \end{rem}

% We now proceed with a few easy observations.
% \begin{prop} \label{prop:prop-sheltered}
%   Let $G$ be a countable, strongly connected, locally finite directed
%   graph and let $W \subseteq V(G)$ be a set of vertices. Then $S^r_G(W) \subseteq \varphi_{G}^r(W)$,
%         in particular $S^r_G(W)$ is finite if $W$ is finite.
% \end{prop}

%\subsection{Finite sheltered hulls  and dead ends in finitely Cayley graphs for finitely generated groups}

Let's now turn our attention back to Cayley graphs.
% as the notion of
%dead end, which we consider in this section, is mostly dealt with in
%this setting.
Let $\Gamma$ be a finitely generated group an let $A$ be a positively generating set containing the unit. For $g \in \Gamma$ we denote %by $|g|_A$ the minimal length of a directed path in %$\Cayley(\Gamma,A)$ from the identity to $g$. Equivalently:
\[|g|_A = d(1_\Gamma,g) = \min \{n \in \mathbb{N}_0 \setsep g \in A^n \}.\]
To simplify notation we denote $S^r_{\Cayley(G,A)}$ and $S_{\Cayley(G,A)}$ by $S_A^r$ and $S_A$.

A \emph{dead end} for $(\Gamma,A)$ is an element $g$ from which
no element of word length larger than $|g|_A$ can be reached
via one of the generators. Equivalently, $g$ is a dead end if and only
if $\varphi_A(\{g\}) \subseteq A^{|g|_A}$. This notion is due to
Bogopolskii, see \cite{bogopolskiiInfiniteCommensurableHyperbolic1997}.
%We say that $g \in \Gamma$ has  \emph{dead end depth at least $n$} in  $\Cayley(\Gamma,A)$ if $\varphi_A^n(\{g\})$ is contained in $A^{|g|_A}$.
We say that $g \in \Gamma$ is a  \emph{dead end of depth at least $n$} for  $(\Gamma,A)$  if $\varphi_A^n(\{g\})$ is contained in $A^{|g|_A}$.

\begin{prop}\label{prop:non-tight-shelter-dead-end}
	If $w \in S_A(A^n) \setminus A^n$, then there is a dead end in $\Gamma \setminus A^n$ of depth at least $|w|_A-n$.
\end{prop}
\begin{proof}
	Let $w \in S_A(A^n) \setminus A^n$. Because $w \not \in A^n$ we have $|w|_A > n$. Because  $w \in S_A(A^n)$  there exists $m >|w|_A-n$
	such that $wA^m\subseteq A^{m+n}$.  
Choose any $v \in A^{n+m-|w|_A}$.
	Then 
		\begin{align*}
		wvA^{|w|_A-n} &\subseteq wA^m \subseteq A^{m+n}=A^{|w|_A+(m+n-|w|_A)} = A^{|w|_A + |v|_A}  \subseteq A^{|wv|_A}. \qedhere
		\end{align*}
Then $wv$ is a dead end of depth at least $|w|_A-n$.
\end{proof}

\begin{cor}\label{cor:infinite-dead-end-depth}
	If $S_A(W)$ is infinite for a finite set $W$, then
	$(\Gamma,A)$ has  dead ends of unbounded depth.
\end{cor}
\begin{proof}
	If $S_A(W)$ is infinite and $W \subseteq A^n$, then
	also $S_A(A^n)$ is infinite. In particular
	there are elements of arbitrary large word length
	in $S_A(A^n)$ and by \Cref{prop:non-tight-shelter-dead-end}
	this implies that $S_A(A^n)$ has dead ends of unbounded depth.
\end{proof}

\begin{cor}\label{cor:finite-dead-end-depth}
  If all dead ends in $(\Gamma,A)$ have  depth at most $C$,
  then $S_A(A^n) \subseteq A^{n+C}$.
\end{cor}
\begin{proof}
  By
  \Cref{prop:non-tight-shelter-dead-end}
  for every $w \in S_A(A^n) \setminus A^n$ we have 
  $|w|_A -n \leq C$ and
  therefore $S_A(A^n) \subseteq A^{C+n}$.
\end{proof}
%\begin{cor}
The following groups are known to have bounded dead end depth (with respect
to any finite \emph{symmetric} generating set, if not mentioned otherwise):
\begin{itemize}
  	\item finitely generated abelian groups \cite[Theorem 1]{lehnertRemarksDepthDead2007} (see \cite[Theorem 3]{sunicFrobeniusProblemDead2008} for the case of the group $\Z$),
			\item hyperbolic groups 
				\cite{bogopolskiiInfiniteCommensurableHyperbolic1997}, see also \cite{warshallDeepPocketsLattices2010}, 
			\item groups with more than one end
				 \cite{lehnertRemarksDepthDead2007},
			\item Thompson's group $F$ with 
				the two standard generators and their inverses
				\cite{clearyCombinatorialPropertiesThompson2004},
			\item Euclidean groups,
				 i.e.\ groups acting discretely by isometries on some $\R^n$ \cite{warshallDeepPocketsLattices2010}.
		\end{itemize}
	Hence in these groups the sheltered hull of every finite set $W \subseteq \Gamma$
	is again a finite set $S_A(W)$.
%\end{cor}

\begin{cor}\label{cor:finite-dead-ends}
	If $(\Gamma,A)$ has only finitely many dead ends,
	then $S_A(A^n)=A^n$ for sufficiently large $n$.
\end{cor}
\begin{proof}
	If $(\Gamma,A)$ has only finitely many dead ends, then
	there is $n \in \N$ such that all dead ends are contained 
	in $A^n$ and then $S_A(A^n)=A^n$ by \Cref{prop:non-tight-shelter-dead-end}.
\end{proof}

\begin{rem}
	While some groups have ``very few'' dead ends,
	 Šunić shows in  \cite[Theorem A.1 3]{sunicFrobeniusProblemDead2008}
	 that every group has a generating 
	 set with respect to which it has at least one dead
	 end. 
\end{rem}

\begin{quest}
 Which finitely generated groups $\Gamma$ have the property that for some finite positively generating set $A$ the sheltered hull  $S_A(W)$ is finite for any finite $W \subseteq \Gamma$? Does this depend on the positive generating set $A$?
\end{quest}

\begin{exam}
Consider the Heisenberg group $H$, given by  the
standard presentation
\begin{align}\label{eqn:Heisenberg}
  H=\left<a,b \,|\, [a,[a,b]], [b,[a,b]] \right>.
\end{align}
This is the simplest  example of a non-abelian finitely generated nilpotent group.
With respect to the generating set $A=\{a,b,a^{-1},b^{-1},e\}$, there are finite sets whose sheltered hull is infinite.
More precisely, any subsets of $H$ which contains $A^2$ has an infinite sheltered hull.
This is a consequence of the following result: If we abbreviate the
commutator of $a$ and $b$ by $c:=[a,b]$, then for any $m \in \N$ there exists $n \in \N$ such that $c^mA^n \subseteq A^{n+2}$. This last result can be extracted (with some additional arguments) for instance using  \cite[Proposition 5.3]{warshallDeepPocketsLattices2010}, which involves a certain ``normal form'' for elements of $H$.
In the light of \Cref{cor:finite-dead-end-depth}
this recovers the known result that $(H,A)$ has dead ends of unbounded depth.
\end{exam}
%\section{Sheltered sets in the Heisenberg group}

The following example shows that the converse of \Cref{cor:finite-dead-end-depth} fails.
More specifically, we show that while the standard lamplighter group
with standard generators has  dead ends of unbounded depth, the sheltered hull of every finite set is  nevertheless
finite.

\begin{exam}
The \emph{Lamplighter group} $L := \Z^2 \wr \Z$ is given by the presentation
\[ L = \left<a,t \,|\, a^2, [a,t^{-n}at^n] \right>.\] This is a
finitely generated solvable group of exponential growth.
The generator ``$a$'' corresponds to ``switching a light'', ``$t$''
corresponds to ``moving the lamplighter''. With respect to the
generating set $A=\{a^{\pm 1},t^{\pm 1},e\}$, $L$ has dead ends of
unbounded depth. Indeed, the element
$$w=  a^{(t^{-n})} \cdot \ldots \cdot  a^{(t^{-1})} \cdot a \cdot a^{(t)} \cdot a^{(t^n)},$$ where $a^{(t^i)}= t^{-i}at^i$, is a  dead end of depth
at least $n+1$. The element $w$ is depicted in \Cref{fig:lamplighter}.
\begin{figure} [h]
	\begin{center}
		\includegraphics[width=.9\textwidth]{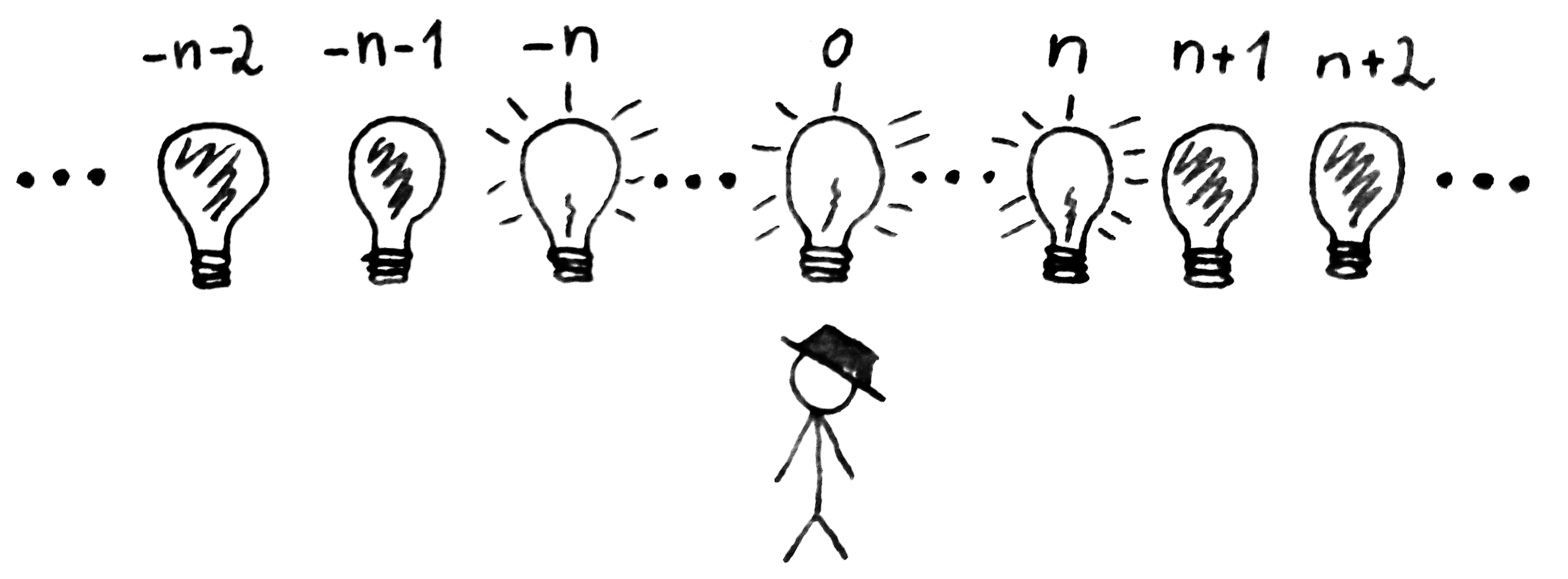}
	\end{center}
	\caption{A dead end in the lamplighter group of depth $n$.}
	\label{fig:lamplighter}
\end{figure}

The element $w$ corresponds to ``the lamplighter  standing at zero with 
lamps from $-n$ to $n$ switched on''. 
Note that  $w$ has length $6n+1$
with respect to the generating set $A$ because 
one needs $n$ steps to move from the origin
to position $n$, $2n$ steps to move to $-n$
and $n$ steps to move back to the origin, and  
additionally there are $2n+1$ lamps to switch on.
Moving the lamplighter within the range from
$-n$ to $n$ only decreases the length of $w$, 
as we don't have to move the lamplighter back to the
origin at the end of the movement. Switching lights
off also decreases the length of $w$, hence
$wA^n \subseteq A^{|w|}$ and $w$ is a dead end of depth
at least $n+1$. 

Nevertheless, the sheltered hull of every finite subset of the lamplighter group $L$ with respect to the generating set $A$ is finite, as shown by the following proposition:

\end{exam}

\begin{prop}
  The sheltered hull of every finite subset of the lamplighter group $L = \Z^2 \wr \Z =\left<a,t \,|\, a^2, [a,t^{-n}at^n] \right>$
  with respect to the
  generators $\{a,t,a^{-1},t^{-1},e\}$ is finite.
\end{prop} 
\begin{proof}
More precisely we will show that
for every $m \in \N$ if $w \in S_A(A^m)$ 
then the lamplighter in $w$
is positioned within $-m, \dots, m$
and there is no lamp switched on outside
of $\{-m, \dots, m\}$.

Assume the lamplighter in $w$
is positioned at $k \in \Z$ 
with $|k| > m$.
Assume wlog.\ that $k>m$.
The element $wa^n$ has
the lamplighter at position $k+n > m+n$,
hence $wa^n \not\in A^{m+n}$ for every $n$.
Thus $w \not\in S_A(A^m)$.

Now assume the lamplighter in
$w$ is positioned at $k \in \{-m,\dots,m\}$ and there is a light switched on
in $w$ at position $\ell$ outside of $\{-m,\dots,m\}$.
If $\ell>0$ consider the element 
$wa^{-n}$. It still has a light switched on at $\ell$
and the lamplighter is positioned at $k-n \leq m <\ell$.
Hence $wa^{-n}$ has length at least $\ell+\ell-(k-n)$,
since the lamplighter has to get to $\ell$ and then back
to $k-n$. But $2\ell-k+n > 2m-m+n=m+n$.
Therefore $wa^{-n} \not \in A^{m+n}$ and $w \not\in S_A(A^m)$.
If on the other hand $\ell<0$, the same reasoning
for $wa^{n}$ leads to $w \not\in S_A(A^m)$.
Therefore $|S_A(A^m)|<(2m+1)2^{2m+1}<\infty$.
Now every finite set $M \subseteq L$ is
contained in $A^m$ for some $m$ and
$S_A(M) \subseteq S_A(A^m)$. 
\end{proof}

\section{Growth related conjugacy invariants for $\varphi_G$}

In this section we will show that we can recover some ``coarse  geometric'' properties of
a  graph $G$ from the dynamics of the map $\varphi_G$. We
begin by showing how to characterize the finite vertex sets
``dynamically''.

\begin{defn}\label{def:fin}
	Suppose  $\varphi:X \to X$ is a function. Define
	\[\Fin(\varphi):=
     \{ x\in X \setsep
     |\varphi^{-r}(\{\varphi^{r+n}(x)\})| < \infty ~\forall r,n \in \N \}.\]
   %For $d >0$ define:
	%$$ g_d(\phi) = \sup_{x \in \Fin(\phi)}\limsup_{n \to
	% \infty}\sup_{r > 0
	% }\frac{\log_2|\phi^{-r}(\{\phi^{r+n}(x)\})}{n^d}.$$
 \end{defn}

\begin{lem}\label{lem:fin-is-fin}
  Let $G$ be a countable, strongly connected locally
  finite directed graph. Then $\Fin(\varphi_G) = \{M \subseteq V(G)
  \setsep |M|<\infty\}$.
\end{lem}
\begin{proof}
  Let $M$ be finite. Then $\varphi_G^{r+n}(M)$ is also finite for every $r,n \in \N$.
  Every subset $N \subseteq V(G)$ with
  $\varphi_G^{r}(N)=\varphi_G^{r+n}(M)$ must be a subset of
  $\varphi_G^{r+n}(M)$, hence there are only finitely many such sets. This shows $M \in \Fin(\varphi_G)$.

  Now let $M \subseteq V(G)$ be infinite. We will show that $|\varphi_G^{-1}(\varphi_G^3(M))| = \infty$, and so $M \not\in \Fin(\varphi_G)$. 
  Since $M$ is infinite and $G$ is locally finite we can choose an infinite $1$-separated set $N \subseteq M$. By this we mean that $N$ is an infinite subset of $M$ such that
  $v_1 \not\in \varphi_G(\{v_2\})$
  for all $v_1,v_2 \in N$ with $v_1 \neq v_2$.
  Define
  \begin{align*}
    M' &:= \left(\varphi_G^2(M) \setminus \varphi_G(N)\right) \cup N,\\
    M''&:= \varphi_G^2(M) \setminus M' = \varphi_G(N) \setminus N.
  \end{align*}
  The sets $M'$ and $M''$ are clearly disjoint.
  From the fact that $N$ is infinite and $1$-separated, it follows that both $M'$ and $M''$ are infinite.
  We will show that for every $U \subseteq M''$, $ U \cup M' \in  \varphi_G^{-1}(\varphi_G^3(M))$. Since $\{U \cup M' \setsep U \subseteq M''\}$ is infinite, this will complete the proof.
  
   Let
  $U \subseteq M''$. Clearly $\varphi_G(U \cup M') \subseteq
  \varphi_G^3(M)$.
  On the other hand
  \begin{align*}
    \varphi_G^3(M)&=\varphi_G((\varphi_G^2(M) \setminus \varphi_G(N)) \cup
                  \varphi_G(N)) \\
                &\subseteq \varphi_G(M') \cup \varphi_G^2(N) \\
    &\subseteq \varphi_G(M') \cup M' \cup \varphi_G(N) \subseteq
      \varphi_G(M') \subseteq
      \varphi_G(M' \cup U).\qedhere
  \end{align*}
\end{proof}

\begin{defn}\label{def:Lk}
  Let $(\Gamma,A)$ be a finitely generated group.
  Define
  \begin{align*}
    \CL^k_A(M) &:= \{N \subseteq \Gamma \setsep
                 \varphi_A^k(N)=\varphi_A^k(M)\} \\
               &\phantom{:}=\varphi_A^{-k}(\{\varphi_A^{k}(M)\}), \\
    L^k_A(M) &:=\log_2 | \CL^k_A(M) |.
  \end{align*}
\end{defn}  

Recall that $A$ is positively generating and contains the identity,
hence  there exists $q \in \N$ such $A^{-1} \subseteq A^q$.

\begin{lem}\label{lem:estimate-number-preimages}
  Let $(\Gamma,A)$ be a finitely generated group and let $q \in \N$ be
  as above.
  Let $M$ be a finite non-empty subset of $\Gamma$ and let 
 $k$ be a positive integer
  such that $M \subseteq A^k$.
  The following inequalities
  hold for every $r \in \N$
  \begin{align*}
    |\varphi_A^{r-1}(M)|
    \leq
    L_A^{(q+1)rk}(\varphi_A^r(M))
    \leq
    |S_A^{(q+1)rk}\left(\varphi_A^r(M)\right)|.
  \end{align*}
\end{lem}
\begin{proof}
  If $\varphi_A^{(q+1)rk}(\{v\}) \subseteq \varphi_A^{(q+1)rk}(\varphi_A^r(M))$
  then by definition $v \in S_A^{(q+1)rk}(\varphi_A^r(M))$.
  This establishes the upper bound on $L_A^{(q+1)rk}(M)$.
  To proof the lower bound on $L_A^{(q+1)rk}(M)$, we verify the following identity.
  \[\varphi_A^{(q+1)rk} (\varphi_A^r(M)) = \varphi_A^{(q+1)rk}
  (\varphi_A^r(M) \setminus \varphi_A^{r-1}(M)).\]

  Clearly the set on the right hand side is contained in the set appearing on the left hand side of
  the identity we wish to check. Since $\Gamma$ is infinite, there is a point $y \in M':=\varphi_A^r(M) \setminus
  \varphi_A^{r-1}(M)$. %\subseteq A^{rk}$.
  Hence $1_\Gamma \in \varphi_{A^{-1}}^{rk}(\{y\}) \subseteq \varphi_A^{q r k}(M')$
  and so $\varphi_A^{r-1}(M) \subseteq \varphi_A^{k+r-1+q r k}(M')
  \subseteq  \varphi_A^{(q+1)rk}(M')$.
  Thus
  \begin{align*}
    \varphi_A^{(q+1)rk}(\varphi_A^r{M})&=\left(\varphi_A^{(q+1)rk+r}(M)\setminus
    \varphi_A^{r-1}(M)\right) \cup \varphi_A^{r-1}(M) \\
  &\subseteq \varphi_A^{(q+1)rk}(\varphi_A^{r}(M) \setminus \varphi_A^{r-1}(M)).\qedhere
  \end{align*}
\end{proof}

We recall notions of strong domination and strong equivalence  for functions on the integers, following \cite[Chapter VI]{harpeTopicsGeometricGroup2000}:

\begin{defn}
  A function $f_1: \N \to \N$ is said
  to \emph{strongly dominate}  a function
  $f_2: \N \to \N$ if
  if there is a constant $C>0$ such that
  $f_1(n)<f_2(Cn)$ for all $n \in \N$.
  Two functions $f_1$ and $f_2$ are
  said to be \emph{strongly equivalent} if they strongly
  dominate each other.
\end{defn}
It is easily verified that strong equivalence  is indeed an equivalence relation.
\begin{defn}
The \emph{growth function} of a finitely generated group $(\Gamma,A)$ is given by
\[r \mapsto |A^r|. \]
\end{defn}

\begin{thm}\label{thm:growth_strongly_equiv}
  Let $(\Gamma,A)$ be a finitely generated group and let $q \in \N$ be such that
  $A^{-1} \subseteq A^q$.
  Let $M \subseteq \Fin(\varphi_A)$ be non-empty and let $k$ and $q$
  be such that $M \subseteq A^k$ and $A^{-1} \subset A^q$. Then the growth function of $(\Gamma,A)$ is strongly equivalent to the function %s \[r \mapsto
  %  |A^r|\]
 % and
  \[r \mapsto L_A^{(q+1)rk}(\varphi_A^r(M)).\]
 % are strongly equivalent.
\end{thm}
\begin{proof}
  Using the left inequality in  \Cref{lem:estimate-number-preimages}  we have
  \[|\varphi_A^{2r-1}(M)|\leq L_A^{2(q+1)rk}(\varphi_A^{2r}(M)).\]
  Since $M$ is non-empty we have $|A^{2r-1}| \le |\varphi_A^{2r-1}(M)|$,
  and as $r \le 2r-1$ it follows that  \[|A^r| \le L_A^{2(q+1)rk}(\varphi_A^{2r}(M)).\]
  On the other hand, using the right inequality in  \Cref{lem:estimate-number-preimages} and
  \[S_A^{(q+1)rk}(M) \subseteq S_A^{(q+1)rk}(A^{r+k}) \subseteq A^{(q+1)rk+r+k},\]
  we conclude that $L_A^{(q+1)rk}(\varphi_A^r(M)) \leq |A^{(q+3)kr}|$.
\end{proof}

\begin{cor}\label{cor:growth_invariant}
  Let $(\Gamma_1, A_1)$ and $(\Gamma_2,A_2)$ be finitely generated groups.
  If $\varphi_{A_1}$ and $\varphi_{A_2}$ are topologically
  conjugate, then $\Gamma_1$ and $\Gamma_2$ have strongly equivalent
  growth functions.
\end{cor}
\begin{proof}
Suppose  $\Phi:\PP(\Gamma_1) \to \PP(\Gamma_2)$ is a topological conjugacy between $\varphi_{A_1}$ and $\varphi_{A_2}$.
Take a non-empty $M_1 \in \Fin(\varphi_{A_1})$, and let $M_2=\Phi(M_1)$. Clearly $M_2 \in \Fin(\varphi_{A_2})$ is non-empty.
Also, for any $r,q,k \in \N$,
\[
L_{A_1}^{(q+1)rk}(\varphi_{A_1}^r(M_1)) =
L_{A_2}^{(q+1)rk}(\Phi(\varphi_{A_1}^r(M_1))) = L_{A_2}^{(q+1)rk}(\varphi_{A_2}^r(M_2))).
\]
Choosing $k$ and $q$ large enough to apply \Cref{thm:growth_strongly_equiv}, it follows that the growth functions of $(\Gamma_1,A_1)$ and $(\Gamma_2,A_2)$ are equivalent.
\end{proof}

When  $\Gamma= \Z^d$ we can say more:
\begin{cor}\label{cor:vol_conv_invariant}
	Let $A_1$ and $A_2$ be positively generating sets of $\Z^{d_1}$ and $\Z^{d_2}$ respectively,
	both containing $0$.
	If $\varphi_{A_1}$ and $\varphi_{A_2}$ are conjugate, then
	$d_1=d_2$ and $\vol_{d_1}(\conv(A_1))=\vol_{d_2}(\conv(A_2))$.
 \end{cor}
 \begin{proof}
   The growth type of $\Z^d$ is $n \mapsto n^d$
   and these growth types are different for pairwise different $d$.
   Recall that by  \cite{lehnertRemarksDepthDead2007} $(\Z^d,A)$ has only finitely many dead ends (the statement in \cite{lehnertRemarksDepthDead2007} is only for symmetric generating sets, but the proof goes through for positively generating sets).
   By \Cref{cor:finite-dead-ends}, 
   we have $S_A(A^n)=A^n$ for sufficiently large $n$  (we use multiplicative notation here for the group operation, even through the group is $\mathbb{Z}^d$).

   Therefore, for $M \in \Fin(\varphi_A)$, $q$ so that $A^{-1}
   \subseteq A^q$, $k$ so that $M \subseteq A^k$ and large $r$,
    \begin{align*}
    \frac1{r^d}|A^{r-1}|
    \leq
    \frac1{r^d} L_A^{(q+1)rk}(\varphi_A^r(M))
      \leq
      \frac1{r^d}|A^{r+k}|.
    \end{align*}
    Sending $r$ to infinity the left and the right side of this
    inequality
    both converge to $\vol(\conv(A))$, a fact which follows directly from  \Cref{prop:iterated_mikowski_sums_convex} below.
    Therefore
       \begin{align*}
    \lim_{d \to \infty}
    \frac1{r^d}\log_2|\varphi_A^{-(q+1)rk}(\{\varphi_A^{(q+1)rk}(\varphi_A^r(M))\})|
         &=\vol(\conv(A)). \qedhere
       \end{align*}
 \end{proof}

\begin{defn}
  For $f: \N \to \N$ we call
  $\omega(f) = \lim_{n \to \infty} \sqrt[n]{f(n)}$ the exponential
  growth rate of $f$ if this limit exists.
\end{defn}

 An argument very similar to that in the proof of \Cref{cor:vol_conv_invariant} shows that 
 for free groups of rank at least $2$ (as well as other hyperbolic  groups) the exponential growth is also ``dynamically recognizable''.
\begin{cor}\label{cor:exp_growth_inv}
  Let $(\Gamma_1,A_1), (\Gamma_2,A_2)$ be two finitely generated groups
  of exponential growth with bounded dead end depth.
  If $\varphi_{A_1}$ and $\varphi_{A_2}$ are topologically
  conjugate, then $\Cayley(\Gamma_1,A_1)$ and $\Cayley(\Gamma_1,A_1)$ have the same exponential
  growth rate.
\end{cor}

For example, the assumptions in \Cref{cor:exp_growth_inv} apply to all hyperbolic groups which are not virtually $\Z$.

\section{Amenability}
Throughout this section  $(\Gamma,A)$ will be a finitely generated
group, and $q \in \N$ will be a constant such that $A^{-1} \subseteq
A^q$.
The aim of this section is to show that amenability of
$(\Gamma,A)$ can be
characterized in terms of the dynamics of $\varphi_{A}$.

Recall that a sequence $(M_n)_{n \in \N}$ of finite subsets of
$\Gamma$ is called a \emph{(right) F{\o}lner sequence} if for any $g
\in \Gamma$ $|M_n g \setminus M_n| / |M_n| \to 0$ as $n \to
\infty$. Existence of a  F{\o}lner sequence is one of many equivalent conditions for amenability of a group \cite{folner1955}. The following lemma states two of many well-known equivalent
conditions for a sequence of subsets to be a  F{\o}lner sequence. See for instance \cite[Chatper $4$]{coornaertCellularBook2010} for details.

\begin{lem}\label{lem:characterization-Folner}
  For an increasing sequence $(M_n)_{n \in \N}$ of finite subsets of
  $\Gamma$ the following are equivalent.
  \begin{enumerate}
  \item $M_n$ is a F{\o}lner sequence.
  \item There is $\ell \in \N$ with $|\varphi_A^{\ell}(M_n)|/|M_n| \to 1$.
  \item There is $\ell \in \N$ with $|\varphi_A^{\ell}(M_n) \setminus M_n| /|M_n| \to 0$.
  \end{enumerate}
\end{lem}

\begin{lem}\label{lem:inner-flexibility}
  Let $k \in \N$ and let $M, N \subseteq \Gamma$ with $\varphi_A^k(M) \subseteq
  \varphi_A^k(N)$. Then
 \[
    \varphi_A^k\left(
    \left(N \cap M\right)
    \cup
    \left(\varphi_A^{(q+1)k}(M) \setminus M\right)
    \right)
    = \varphi_A^k\left(\varphi_A^{(q+1)k}(M)\right).
\]
\end{lem}
\begin{proof}
Clearly
 \[
    \varphi_A^k\left(
    \left(N \cap M\right)
    \cup
    \left(\varphi_A^{(q+1)k}(M) \setminus M\right)
    \right)
   \subseteq \varphi_A^k\left(\varphi_A^{(q+1)k}(M)\right).
\]
Let us show the other inclusion.
Suppose  $u \in \varphi_A^k(\varphi_A^{(q+1)k}(M))$. We need to show that 
$u \in \varphi_A^k\left(
    \left(N \cap M\right)
    \cup
    \left(\varphi_A^{(q+1)k}(M) \setminus M\right)
    \right)$. If   $u \in \varphi_A^k(\varphi_A^{(q+1)k}(M) \setminus M)$, we are done. So   assume that
  $u \not\in \varphi_A^k(\varphi_A^{(q+1)k}(M) \setminus M)$. 
%We have:
%\[
%\varphi_A^k(\varphi_A^{\ell}(M))
%  \setminus \left(\varphi_A^k(\varphi_A^{\ell}(M) \setminus M)\right) \subseteq \varphi_A^k(M).
%\]
Thus,
\[
u \in  \varphi_A^k(M)
  \setminus \left(\varphi_A^k(\varphi_A^{(q+1)k}(M) \setminus M)\right).
\]
In particular, because $\varphi_A^{k}(M)\subseteq \varphi_A^{k}(N)$,
$u \in \varphi_A^{k}(N)$, so there exists $v \in N$ so that $u  \in
\varphi_A^k(\{v\})$. Hence $v \in \varphi_A^{qk}(\{u\}) \subseteq \varphi_A^{(q+1)k}(M)$.
But the assumption  $u \not\in \varphi_A^k(\varphi_A^{(q+1)k}(M)
\setminus M)$ together with $u \in \varphi_A^{k}(\{v\})$ implies that $v \in M$.  
We conclude that
$u \in \varphi_A^k( N \cap M)$.
This completes the proof.
\end{proof}

\begin{lem}\label{lem:Lk-upper-bound}
  Let $M \subseteq  \Fin(\varphi_A)$ and $k \in \N$.
  If $\ell \geq (q+1)k$  then
  \[L_A^k(\varphi_A^\ell(M)) \leq L_A^k(\varphi_A^{(q+1)k}(M)) +
|\varphi_A^{\ell+k}(M) \setminus M|.\]
\end{lem}
\begin{proof}
  Consider the map
  \begin{align*}
    \Psi&:  \CL^k_A\left(\varphi_A^{(q+1)k}(M)\right)
 \times \PP\left(\varphi_A^{\ell+k}(M) \setminus
    M\right)     \to \varphi_A^{\ell+k}(M), \\
    \Psi(W_1,W_2)&:=W_1 \triangle W_2.
  \end{align*}

\def\firstcircle{(0,0) circle (1cm)}
\def\secondcircle{(0,0) circle (2cm)}
\def\thirdcircle{(2,0) ellipse (1.5cm and 0.8cm)}

% Now we can draw the sets:
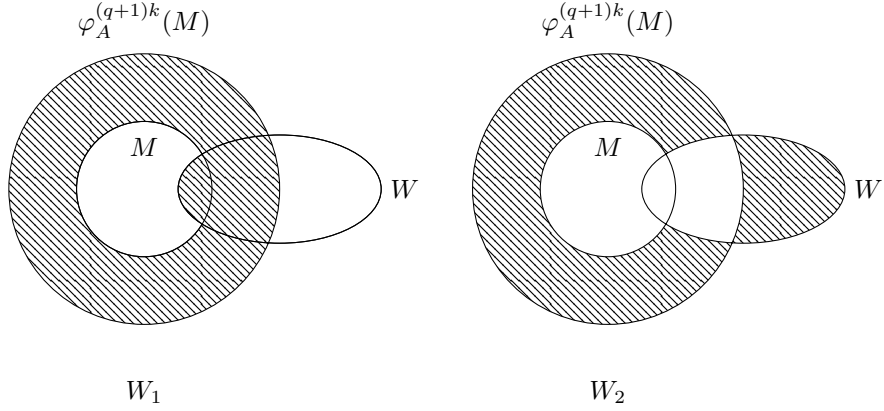
\begin{figure}
  \begin{center} \begin{tikzpicture}[scale=0.9]  
    \fill[pattern=north west lines] \secondcircle;
    \fill[draw=black, fill=white] \firstcircle;
    \fill[draw=black, fill=white] \thirdcircle;
    \draw \secondcircle \firstcircle \thirdcircle;
  	\begin{scope} 
	\clip \secondcircle;
	\fill[pattern=north west lines] \thirdcircle;
 \end{scope}
   \draw (0,2.1) node[above]{$\varphi_A^{(q+1)k}(M)$};
   \draw (0,0.9) node[below] {$M$};
	\draw (3.5,0) node [right] {$W$};
   \draw (0,-3) node {$W_1$};
 \end{tikzpicture} \hspace{1em}
\begin{tikzpicture}[scale=0.9]
    \fill[pattern=north west lines] \secondcircle;
    \fill[fill=white] \firstcircle;
    \fill[pattern=north west lines] \thirdcircle;
   \begin{scope} 
	\clip \secondcircle;
	\fill[white] \thirdcircle;
   \end{scope}
   \draw \secondcircle \firstcircle \thirdcircle;
  	\draw (0,2.1) node[above]{$\varphi_A^{(q+1)k}(M)$};
   \draw (0,0.9) node[below] {$M$};
	\draw (3.5,0) node [right] {$W$};
   \draw (0,-3) node {$W_2$};
 \end{tikzpicture}
\end{center}
\caption{Illustration of the decomposition of $W$ in the proof of
  \Cref{lem:Lk-upper-bound}.}
\label{fig:decomp-upper-bound}
\end{figure}

  We have to show that $\CL_A^k(\varphi_A^\ell(M))$ is contained in
  the image of $\Psi$. Let $W \in \CL_A^k(\varphi_A^\ell(M))$.
  Then (see \Cref{fig:decomp-upper-bound})
  \begin{align*}
    W &= W_1 \triangle W_2 \text{ with} \\
    W_1&:=\left(\left(W \cap M\right) \cup \left(\varphi_A^{(q+1)k}(M)
         \setminus M\right)\right),\\
    W_2&:=
         \left(\varphi_A^{(q+1)k}(M) \triangle (M \cup W)\right). 
  \end{align*}
  By \Cref{lem:inner-flexibility} the set $W_1$ is contained in
  $\CL_A^k(\varphi_A^{(q+1)k}(M))$.
  Since $W \in \CL_A^{k}(\varphi_A^\ell(M)) \subseteq
  \PP(\varphi_A^{\ell+k}(M))$ and $W_2 \cap M = \emptyset$
  we also have $W_2 \subseteq \varphi_A^{\ell+k}(M)\setminus M$.
\end{proof}

\begin{lem}\label{lem:bound-vertex-cover}
  Let $H$ be a finite undirected graph  such that
  every connected component contains at least two vertices.
  Then $H$ has a vertex cover of size at most $|V(H)|/2$.
\end{lem}
\begin{proof}
  Choose a spanning tree in every connected component. Since every one
  of these spanning trees is bipartite, we can pick the smaller of the two
  partition classes in each of them. Since all our spanning trees
  contain at least two vertices, the picked vertices form a vertex cover
  and we picked at most half of the vertices.
\end{proof}

\begin{lem}\label{lem:Lk-is-size}
  Let $M \subseteq  \Fin(\varphi_A)$.
  Then for all $k \geq q$, $\ell >1$ we have
  $L_A^k(\varphi_A^\ell(M)) \geq \frac{1}{2} |M|$.
\end{lem}
\begin{proof}
  Form the following graph $H$. The vertices of $H$ are the elements
  of $\varphi_A(M)$ and we add an edge between $u$ and $v$ in $\varphi_A(M)$
  if $u \in \varphi_A^{q}(\{v\})$ and $v \in \varphi_A^{q}(\{u\})$.
  Since for every $a \in A$ and $u \in M$ we have $u=uaa^{-1} \in
  \varphi_{A^{-1}}(\{ua\}) \subseteq \varphi_{A}^q(\{ua\})$ and $ua \in \varphi_A(\{u\})$,
  every connected component of $H$ contains at least two vertices.
  Hence by \Cref{lem:bound-vertex-cover} we can
  find a subset $W \subseteq \varphi_A(M)$ with $|W| \leq \frac12 |\varphi_A(M)|$ such that for
  all elements $v \in \varphi_A(M)$ there is a vertex $w \in W$ with $v \in
  \varphi_{A}^q(\{w\})$, hence $\varphi_A(M) \subseteq \varphi_A^{k}(W)$.
  Thus for every set $N \subseteq \varphi_A^{\ell}(M)$
  with $\left(\varphi_A^{\ell}(M) \setminus \varphi_A(M)\right) \cup W \subseteq N$
  we have $\varphi_A^{k}(N) = \varphi_A^{k}(\varphi_A^\ell(M))$. Since there
  are $2^{|\varphi_A(M)|-|W|}$ such sets, we have
  $L_A^k(\varphi_A^\ell(M)) \geq \frac{1}{2}|\varphi_A(M)|\geq \frac{1}{2}|M|$.
\end{proof}

\begin{lem}\label{lem:Lk-lower-bound}
  Let $k \in \N$ and  let $M \subseteq  \Fin(\varphi_A)$.
  If $\ell\geq (q+2)k+3$ then 
\[L_A^k(\varphi_A^\ell(M)) \geq L_A^{k}(\varphi_A^{(q+1)k}(M))
  +|\varphi_A^{\ell-1}(M) \setminus \varphi_A^{\ell-2}(M)| .\]
\end{lem}
\begin{proof}
By definition of $L_A^k$, we need to prove that
\[
|\CL^k(\varphi_A^\ell(M))| \ge \left| \CL^k_A(\varphi_A^{(q+1)k}(M)) \times
 \PP\left(\varphi_A^{\ell-1}(M) \setminus
    \varphi_A^{\ell-2}(M)\right) 
\right|.
\]
  We prove this by constructing an injective function
\[
    \Psi: \CL^k_A(\varphi_A^{(q+1)k}(M)) \times \PP\left(\varphi_A^{\ell-1}(M) \setminus
    \varphi_A^{\ell-2}(M)\right) 
  \to \CL_A^k(\varphi_A^\ell(M)).\]
This is given by:
 \[  \Psi(Q,P):= \left(\varphi_A^{\ell}(M) \setminus
    \varphi_A^{\ell-1}(M)\right) \cup P \cup \left(\varphi_A^{\ell-2}(M)
              \setminus \varphi_A^{(q+2)k}(M)\right)\cup Q.
  \]
  First we check that
  the image of this map lies indeed in $\CL_A^k(\varphi_A^\ell(M))$.
  Let $(Q,P) \in  \CL^k_A(\varphi_A^{(q+1)k}(M)) \times \PP\left(\varphi_A^{\ell-1}(M) \setminus
    \varphi_A^{\ell-2}(M)\right)$.
  It is clear that $\varphi_A^k(\Psi(Q,P)) \subseteq
  \varphi_A^{k+\ell}(M)$.
  We also have \[\left(\varphi_A^{\ell+k}(M) \setminus
    \varphi_A^{\ell-1}(M) \right) \cup \left(\varphi_A^{\ell-2+k}(M)
  \setminus \varphi_A^{(q+2)k}(M)\right) = \varphi_A^{\ell+k}(M) \setminus
  \varphi_A^{(q+2)k}(M) \subseteq \Psi(Q,P).\]
  Finally $\varphi_A^k(Q)=\varphi_A^{(q+2)k}(M)$,
  hence $\varphi_A^{k+\ell}(M) \subseteq \varphi_A^k(\Psi(Q,P))$.
  
  It is now enough to check that $\Psi$ is injective.
  This follows from
  \begin{align*}
    P &=\Psi(Q,P) \cap \left(\varphi_A^{\ell-1}(M)\setminus
        \varphi_A^{\ell-2}(M)\right), \\    
    Q&=\Psi(Q,P) \cap \varphi_A^{(q+2)k}(M). \qedhere
  \end{align*}
\end{proof}

\begin{thm}\label{thm:charact-amenability}
  Let $(\Gamma,A)$ be a finitely generated group and let $q \in \N$ be such that
  $A^{-1} \subseteq A^q$.
  Then $\Gamma$ is amenable if and only if
  there is a sequence of finite sets $(M_n)_{n \in \N}$
  in $\Fin(\varphi_A)$ such that
  \begin{align}
    \lim_{n \to \infty}\frac{L^q_A(\varphi_A^{(q+5)q}(M_n))}{L^q_A(\varphi_A^{(q+1)q}(M_n))}
    =1.\label{eq:chara-amenability}
  \end{align}
\end{thm}
\begin{proof}
  Let $\Gamma$ be amenable and let $(M_n)_{n \in \N}$
  be a F{\o}lner sequence.
  By \Cref{lem:Lk-upper-bound} and \Cref{lem:Lk-is-size} we have
  \begin{align*}
    1 \leq \frac{L^q_A(\varphi_A^{(q+5)q}(M_n))}{L^q_A(\varphi_A^{(q+1)q}(M_n))}
    &\leq 1 +
      \frac{
      |\varphi_A^{(q+6)q}(M_n) \setminus M_n|
      }{
      L^q_A(\varphi_A^{(q+1)q}(M_n))} \\
    &\leq 1 + \frac{
      2|\varphi_A^{(q+6)q}(M_n) \setminus M_n|
      }{
      |M_n|},
  \end{align*}
  where the last inequity follows from \Cref{lem:Lk-is-size}.
  By \Cref{lem:characterization-Folner} the right hand side converges
  to one and therefore \eqref{eq:chara-amenability} is satisfied.
  
  On the other hand assume that $(M_n)_{n \in \N}$
  is a sequence in $\Fin(\varphi_A)$ satisfying
  \eqref{eq:chara-amenability}. By \Cref{lem:Lk-lower-bound} we have
  \begin{align*}
     \frac{L^q_A(\varphi_A^{(q+5)q}(M))}{L^q_A(\varphi_A^{(q+1)q}(M))}
    &\geq 1 +
      \frac{
      |\varphi_A^{q(q+5)-1}(M_n)\setminus \varphi_A^{q(q+5)-2}(M_n)|
      }{
      L^q_A(\varphi_A^{(q+1)q}(M_n))
      } \\
    &\geq 1 +
      \frac{
      |\varphi_A^{q(q+5)-1}(M_n)\setminus \varphi_A^{q(q+5)-2}(M_n)|
      }{
      |\varphi_A^{(q+2)q}(M_n)|
      }  \\
    &\geq 1 +
      \frac{
      |\varphi_A^{q(q+5)-1}(M_n)\setminus \varphi_A^{q(q+5)-2}(M_n)|
      }{
      |\varphi_A^{(q+5)q-2}(M_n)|
      }  \geq 1
  \end{align*}
  Since we assumed that the left side of this inequality converges to
  one, this shows that
  \begin{align*}
  \frac{
      |\varphi_A\left(\varphi_A^{(q+5)q-2}(M_n)\right)\setminus \varphi_A^{(q+5)q-2}(M_n)|
      }{
      |\varphi_A^{(q+5)q-2}(M_n)|
      } \to 0
  \end{align*}
  Hence by \Cref{lem:characterization-Folner}
  the sequence $(\varphi_A^{(q+5)q-2}(M_n))_{n \in \N}$ is a F{\o}lner sequence and
  $\Gamma$ is amenable.
\end{proof}

\begin{cor}\label{cor:amenablity_recognizable}
  Let $(\Gamma_1,A_1)$ and $(\Gamma_2,A_2)$ be finitely generated groups
  such that $\varphi_{A_1}$ and
  $\varphi_{A_2}$
  are topologically conjugate. If $\Gamma_1$
  is amenable, then $\Gamma_2$ is amenable too.
\end{cor}

\section{North-South dynamics}\label{sec:north_south}
\begin{defn}
  A homeomorphism $T:X \to X$ of a compact metric space $X$ has
  \emph{north-south dynamics} if there are precisely two fixed points
  $x^+,x^- \in X$ for $T$ such that $\lim_{n \to \infty}T^n(y)=x^+$
  for every $y \in X \setminus \{x^-\}$ and
  $\lim_{n \to \infty}T^{-n}(y)=x^-$ for every
  $y \in X \setminus \{x^+\}$.
\end{defn}

Two simple examples of homeomorphisms with north-south dynamics are the
map $t \mapsto \sqrt{t}$ on the interval $[0,1]$ and the map
$n \mapsto n+1$ on the two-point compactification $\Zinfty= \Z \cup \{+\infty,-\infty\}$.

For any homeomorphism $T:X \to X$ the map $S(T):S(X) \to S(X)$ given
by $S(T)(x,t)=(T(x),\sqrt{t})$ has north-south dynamics.  Here $S(X)$
is the suspension of the topological space  $X$ given by 
\[S(X) = \left( X \times [0,1] \right) / \sim  \] 
\[ (x_1,t_1) \sim (x_2,t_2) ~\Leftrightarrow t_1=t_2= 0 \mbox{ or } t_1=t_2 = 1 \mbox{ or } (x_1,t_1)=(x2,t_2).\]
 So for instance the
suspension of the $d$-dimensional sphere is the $d+1$-dimensional
sphere.

In a similar way we can define $S_\Z(T): S_\Z(X) \to S_\Z(X)$ by
$S_\Z(T)(x,n)=(T(x), n+1)$ where $S_{\Z}(X)$ is a disconnected analog of the suspension given by:
\[S_{\Z}(X) = \left( X \times \Zinfty\right) / \sim  \] 
\[ (x_1,n_1) \sim (x_2,n_2) ~\Leftrightarrow n_1=n_2=+\infty \mbox{ or } n_1=n_2 = -\infty \mbox{ or } (x_1,n_1)=(x_2,n_2).\]
%is obtained by
%collapsing the sets $X \times \{\infty\}$ and $X \times \{-\infty\}$
%in $X \times \Z_{\pm \infty}$.
 A complete characterization
of north-south systems was obtained for many spaces, see
\cite{levittNorthSouthHomeomorphismsSierpinski2004} for a survey of
known results. This includes the Cantor set for which
the following uniqueness result was obtained in \cite{levittHomeomorphismesDynamiquementSimples1998}.
We include a short proof for self-containment.

\begin{prop}\label{prop:north_south_cantor_unique}
  Up to topological conjugacy there is a unique homeomorphism with
  north-south dynamics on the Cantor space
  $\{0,1\}^{\mathbb{N}}$.
\end{prop}
\begin{proof}
We will show that any  north-south dynamics on a Cantor space is topologically conjugate to the  ``standard''  
 north-south dynamics $\varphi(x,n)= (x,n+1)$, $\varphi:S_\Z(C) \to  S_{\Z}(C)$ where
$C$ is a Cantor space and $S_{\Z}(C)$ is the ``disconnected
suspension'' defined above.
It is easy to check that $S_\Z(C)$ is compact, totally disconnected, second countable and has no isolated points, so it is a Cantor space.
 
Let $X$ be a Cantor space and $T: X \to X$ be any homeomorphism with
  north-south dynamics.
  Let $D$ be an clopen neighborhood of the unique positively attracting fixed point $x^+$ such that $X \setminus D$
  is a neighborhood of the unique negatively attracting fixed point $x^-$.
  For every point $x \in X \setminus \{x^+,x^-\}$
  the expression
  $n_D(x)=\sup \{n \in \Z \setsep T^n(x) \in X \setminus D\}$ is a well defined integer.
  We want to show that \[C:=\{x \in X \setminus \{x^+,x^-\} \setsep
  n_D(x)=0\} = (X \setminus D) \cap \bigcap_{k=1}^\infty T^{-k}(D)\]
  is a clopen set in $X$.
  For every point $x \in X \setminus
  \{x^+\}$ there is $k >0$ such that $T^{-k}(x) \in X \setminus D$,
  hence $\bigcap_{k=1}^\infty T^{k}(D) \cap (X \setminus D)$
  is empty. By compactness there must be $N \in \N$
  such that already the finite intersection
  $\bigcap_{k=1}^{N} T^{k}(D) \cap (X \setminus D)$
  is empty. In other words, $\bigcap_{k=1}^{N} T^{k}(D) \subseteq  D$
  and therefore $\bigcap_{k=1}^N T^{-k}(D) \subseteq T^{-(N+1)}(D)$.
  Thus $\bigcap_{k=1}^N T^{-k}(D)=\bigcap_{k=1}^{N+1} T^{-k}(D)
  \subseteq T^{-(N+2)}(D)$. By induction we obtain
  $\bigcap_{k=1}^N T^{-k}(D)=\bigcap_{k=1}^{\infty} T^{-k}(D)$.
  This shows that $C= (X \setminus D) \cap \bigcap_{k=1}^{\infty} T^{-k}(D)$
  is indeed a clopen subset of $X$ hence a Cantor space itself.
  Notice that the orbit of every non fixed point $x$ of $T$
  intersects $C$ in precisely one point, namely $T^{n_D(x)}(x)$,
  since $n_D(T(x))=n_D(x)-1$.
  
  Let $y^+$ be the attracting and $y^-$ the repelling fixed
  point of $\varphi$. We define  $\Phi: X \to S_\Z(C)$
  by
\[
\Phi(x)= \begin{cases}
y^+ & \text{ if }x=x^+\\
y^- & \text{ if }x=x^-\\
(T^{n_D(x)}(x), -n_D(x)) & \text{ otherwise}.
\end{cases}
\]
% $x^+ \mapsto y^+, x^- \mapsto y^-$ and 
 % \begin{align}
 %   x \mapsto (T^{n_c(x)}(x), -n_c(x)).
 % \end{align}
  It is clear that $\Phi$ intertwines $\varphi$ and $T$ and
  that it is bijective. To complete the proof, we check that $\Phi$ is continuous.

  Let $(x_n)_{n \in \N}$ be a sequence of points in $X$ so that  $\lim_{n \to \infty}x_n=x \in X$.
  If $x \not \in \{x^-,x^+\}$ then $x \in T^{-n_D(x)}(C)$.
  Since $T^{-n_D(x)}(C)$ is open, it also contains
  $x_n$ for sufficiently large $n$.
  For these $n$ we have
  $\Phi(x_n)=(T^{-n_D(x_n)}(x_n),-n_D(x_n))=(T^{-n_D(x)}(x_n),-n_D(x))
  \to \Phi(x)$.
  If $x=x^+$ then $n_D(x_n) \to -\infty$ and $\Phi(x_n) \to y^+$.
  Similarly,  if $x=x^-$ then $n_D(x_n) \to \infty$ and $\Phi(x_n)
  \to y^-$.
\end{proof}

\section{The eventual image and the natural extension of $\varphi_G$}\label{sec:eventual_image_natural_extension}

Let $(X,\varphi)$ be a topological dynamical
system,  not necessarily invertible . Namely $X$ is a compact topological space and
$\varphi:X \to X$ is a continuous self-map.  The \emph{eventual image}
(also called the \emph{maximal attractor}) is given by
$\Evt(\varphi):=\bigcap_{n=1}^\infty \varphi^n(X)$.
	\begin{prop}\label{prop:maximal_attractor}
     For a non-empty set $M \subseteq V(G)$ with $M \neq V(G)$ the following are equivalent:
     \begin{enumerate}
     \item$M$ is in the eventual image $\bigcap_{n=1}^\infty
         \varphi_{G}^n(\PP(V(G)))$
         \label{enum:max_attractor-a} 
     \item $M$ is the union of arbitrary large balls, i.e.
       for every $r$ there is $M' \subseteq M$ such that
       $M=\varphi_G^r(M')$.
       \label{enum:max_attractor-b}
     \item $M$ is a union of horoballs.
       \label{enum:max_attractor-c}
     \item $M$ is the union of Busemann balls.
       \label{enum:max_attractor-d}  
     \end{enumerate}
	\end{prop}
	\begin{proof}
     Condition \eqref{enum:max_attractor-a} and
     \eqref{enum:max_attractor-b} are obviously equivalent
     as for any $M' \subseteq V(G)$ and $n >0$
     we have $\varphi_G^n(M')=\bigcup_{w \in M'} \varphi_G^n(\{w\})$.
     Assume now that $M$ is a union of arbitrary large balls.
     For every $v \in M$ there is a sequence of
     points $w_k$ and increasing radii $r_k$ such that
     $v \in \varphi_G^{r_k}(\{w_k\}) \subseteq M$.
     Taking a limit along a subsequence, one obtains
     a horoball in $M$ containing $v$.
     That every union of horoballs is a union of Busemann balls
     and vice versa follows directly from \Cref{thm:Busemannball}. Condition \eqref{enum:max_attractor-d} implies \eqref{enum:max_attractor-b} because every Busemann ball is an increasing union of  arbitrary large balls.
	\end{proof}
In view of the above proposition, we call elements of $\Evt(\varphi_G)$ \emph{horoballunions}.

The \emph{natural extension} of a dynamical system $(X,\varphi)$ (which is not necessarily injective nor surjective), 
 is the dynamical system
$(\hat X_\varphi,\hat \varphi)$ where
%  \varprojlim \varphi := 
$$\hat X_\varphi :=\left\{ \hat x \in X^\Z \setsep \hat x_{n+1} = \varphi(\hat x_{n}) \mbox{ for all } n \in \Z\right\},$$
and $\hat\varphi:\hat X_\varphi  \to \hat X_\varphi$ is the shift given by
\[ \hat \varphi(\hat x)_n = \varphi( \hat x_{n}).\]
Note that $\hat \varphi(\hat x)_n = \hat x_{n+1}$ for every
$x \in \hat X_\varphi$ and $n \in \Z$.  It follows that
$\hat \varphi:\hat X_\varphi \to \hat X_\varphi$ is a homeomorphism.  The natural
extension factors onto the eventual image via the projection
$\hat x \mapsto \hat x_0$. Any other invertible extension of the
eventual image factors through the natural extension.

For a countably infinite graph $G$, the natural extension of $(\PP(V(G)),\varphi_G)$  is topologically
  conjugate to the map $  (x_v)_{v \in V(G)} \mapsto (x_v + 1)_{v \in V(G)}$ on the space 
$\tilde  X_G \subseteq \Zinfty^{V(G)}$, consisting of the two points  $x^+ :=(+\infty)^{V(G)}$ and $x^- := (-\infty)^{V(G)}$  and all the points $x \in \Z^{V(G)}$ that satisfy:
\begin{itemize}
\item  For every $v \in V(G)$ there exists $w \in V(G)$ such that $(w,v) \in E(G)$ and $x_v = x_w +1$ and
\item For every $(v,w) \in E(G)$ $x_w \le x_v +1$.
\end{itemize}
Informally we can say that the natural extension of
$\varphi_G:\PP(V(G)) \to \PP(V(G))$ is  ``pointwise incrementing by $1$ on the two-point compactification of
the integer valued 1-Lipschitz functions on $V(G)$ without local maxima''.
With  this identification, the natural extension of the subsystem
corresponding to the  closure of the horoballs consists of $\Z$-valued
functions which are vertical translations of horofunctions (with the
two additional points $x^-$ and $x^+$ ``at infinity''), see \Cref{lem:horofunction-sublevel}.

\begin{prop}\label{prop:perfect_NA_implies_unique}
  Suppose $G_1$,$G_2$ are both countable, strongly connected locally
  finite directed graphs. Let $(\hat X_{G_i},\hat \varphi_{G_i})$
  denote the natural extensions of $\varphi_{G_i}$ for $i=1,2$. If $\hat X_{G_1}$ and $\hat X_{G_2}$ have no isolated
  points, than $(\hat X_{G_1},\hat \varphi_{G_1})$ is topologically
  conjugate to $(\hat X_{G_2},\hat \varphi_{G_2})$.
\end{prop}
\begin{proof}
  For any strongly connected, countable graph $G$ the natural extension of $\varphi_{G}$
  has a unique attracting fixed point $x^+$ given by $x^+_n =V(G)$ for
  every $n \in \Z$, and a unique repelling fixed point $x^-$ given by
  $x^-_n=\emptyset$. The natural extension of $\varphi_{G}$ acts on a
  closed subspace of $\PP(V(G))^\Z$, so under the assumption of
  no isolated points it has north-south dynamics on the Cantor set.
  The result follows by \Cref{prop:north_south_cantor_unique} .
\end{proof}
 
\section{The geometry of horoballs in $\Z^d$}\label{sec:Z_d_geo_horoballs}

In the following sections we will exclusively consider abelian groups ($(\Z^d,+)$ and $(\R^d,+)$), so we switch to additive notation.
For $B,C \subseteq \Z^d$ or $B,C \subseteq \R^d$ we write $B+C$ for the Minkowski sum
$$B+C = \left\{b+c \setsep b \in B,~ c \in C \right\}.$$
For $B \subseteq \Z^d$ or $B \subseteq \R^d$ and $n \in \N$ we abbreviate the $n$'th-fold  Minkowski sum of $B$ by $nB$:
$$n B = \underbrace{B+\ldots +B}_n= \left\{b_1 + \ldots + b_n \setsep b_1,\ldots,b_n \in B \right\}.$$
Note that whenever $B \subseteq \R^d$ is convex, the $n$-fold
Minkowski sum coincides with $n$-dilation, meaning $nB = \{nb \setsep
b \in B\}$ for $n \in \N$.
We denote the ball of radius $R$ in $\R^n$ around $v$ by $B_R(v)$. We denote the convex hull of a set $B \in \R^d$ by $\conv(B)$, and by $\Ext(B)$ the extremal points (or vertices) of $\conv(B)$.
Whenever we write $\conv(B)$ for a set $B \subset \Z^d$, we mean the convex hull of $B$  in $\R^d$.

In the following,  $A$ will be a positively generating set of the additive group
$\Z^d$. The set $A$  will not necessarily be symmetric, but
it will always contain $0$. This last assumption is mostly for
convenience, because for any positively generating set $A$,  $0 \in nA$  along an arithmetic progression of positive integer $n$'s.

Iterated  Minkowski sums  of a finite positively generating set of
$\Z^d$ are ``roughly'' equal to the integer points in the dilated
convex hull. A proof of this fact appears in \cite{MR2988074}.
For completeness, we include a precise statement and a short proof based on the well known Shapley-Folkman Lemma (see for instance \cite{MR0385711}).
See \cite{willsonConvergence1978} for a similar convexity-based proof
of a closely related result, and also \cite{gaubertTropical2018}.

\begin{prop}\label{prop:iterated_mikowski_sums_convex}
	Let $A \subseteq \Z^d$ be a finite positively generating set with $0 \in A$. 
	Then for any $n \in \N$ 
	$$n A \subseteq n\conv(A) \cap \Z^d,$$
	and there exists 
	$N \in \N$ so that for every  $n >N$
	$$(n-N)\conv(A) \cap \Z^d \subseteq n A.$$
\end{prop}
\begin{proof}
  By definition of the convex hull, $A \subseteq \conv(A)$, so $nA \subseteq n\conv(A)
  \cap \Z^d$.
  To prove the second part we apply
  the Shapley-Folkman Lemma which implies that for any compact
  $B \subseteq \R^d$ and $m \ge d$
	
	\[\conv(mB) = (m-d)B + d\conv(B).\] 
	Since $A \subset \Z^d$, it follows that
	\[	\conv(mA) \cap \Z^d= (m-d)A + (d\conv(A) \cap \Z^d).
	\]
	Since $A$ is positively generating and contains the identity, the
   sequence $kA$ monotonically increases to $\Z^d$ as $k \to \infty$.
   Since $d\conv(A) \cap \Z^d$ is a finite subset of $\Z^d$, it
   follows that there exists $k \in \N$ so that
   $d\conv(A) \cap \Z^d \subseteq kA$.  Together this implies
	\[m\conv(A) \cap \Z^d=\conv(mA) \cap \Z^d \subseteq (m-d)A + kA.\qedhere\]
\end{proof}

\begin{defn}\label{defn:M_F}
	For a face $F$ of $\conv(A)$ let $M_F$  denote the collection of  $(d-1)$-dimensional faces $F'$ of $\conv(A)$ that contain $F$.
\end{defn}

\begin{defn}\label{def:envelope}
	%	\begin{equation}
	%	\mathit{Spt}(F,A) = \left\{ w \in \R^d \setsep \langle u,w\rangle \le \langle v,w \rangle ~ \forall v \in F,~ u \in A \right\}.
	%	\end{equation}
	For  a  $(d-1)$-dimensional face $F$ of $\conv(A)$ let $\ell_F$ denote the inward facing tangent vector of $F$, namely the unique $\ell_F \in \mathbb{R}^d$ satisfying
	\begin{itemize}
		\item $\| \ell_F\|=1$,
		\item $\langle  \ell_F ,u -v \rangle
		=0$  for all $u,v \in F$,
		\item $\langle  \ell_F ,u -v \rangle
		\ge 0$  for all $u \in \conv(A)$, $v \in F$.
	\end{itemize}
	Given a face $F$ of $\conv(A)$, we define the \emph{envelope of $A$ with
	respect to $F$} to be
	\begin{equation}\label{eq:Env_exp}
	\Env_F =\Env_{F,A}:=% \left\{ v\in \R^d \setsep \langle v,w \rangle \le 0
	\bigcap_{F' \in M_F}
	\left\{  v\in \R^d \setsep
	\langle  \ell_{F'},v \rangle  \ge 0
	\right\},
	\end{equation}
	%where the intersection is over all the $(d-1)$-dimensional  faces $F'$  of $\conv(A)$ that %contain $F$.	
\end{defn}
It follows from the definition that
\begin{equation}\label{eq:Env_alt_exp}
\Env_F =
\left\{\sum_{v \in \Ext(F)}\sum_{u \in A}\alpha_{u,v}(u-v) \setsep
\alpha_{u,v} \in [0,\infty) ~\forall u \in A ,\, v \in \Ext(F)
\right\}. 
\end{equation}

In particular,
if $F=\{v\}$ is a zero dimensional face of $\conv(A)$, then 
$$\Env_F = \left\{\sum_{u \in A}t_u(u-v) \setsep t \in [0,\infty) \right\}.$$
Also, 
if $F$ is a $(d-1)$-dimensional face of $\conv(A)$, then $\Env_F$ is a translation of the unique  half-plane containing $A$ whose boundary is the affine hull of $F$.

In $\Z^d$ the structure of $A$-horoballs is essentially given by the convex hull of $A$.
There is a one-to-one correspondence between faces of $\conv(A)$ and  $A$-horoballs up to translation. 
The following theorem makes this precise: 

\begin{thm}\label{thm:ZD_horoballs_are_faces_of_convex_hull}
	Let $A \subseteq \Z^d$ be a finite generating set.
	For any face $F$ of $\conv(A)$ there exists a unique $A$-horoball $H_F \subseteq \Z^d$ 	with the property that $0 \in H_F$ and
	\[ \Env_F = \conv(H_F).\]
	%	 Furthermore,
	% there exists $v \in \Z^d$ so that 
	%	\begin{equation}\label{eq:H_F_Env}
	%	v+ \Env_F \cap \Z^d \subseteq H_F \subseteq \Env_F \mbox{ for some }v \in \Z^d.
	%	\end{equation}  
	The horoball $H_F$ is given by
	\begin{equation}\label{eq:H_F_module}
	H_F = \sum_{w \in \Ext(F)} \bigcup_{j =0}^\infty j(A- w).
	\end{equation}
	
	Conversely, any $A$-horoball is of the form $v+ H_F$ for some $v \in \Z^d$ and a face $F$ of $\conv(A)$. This face is uniquely determined.
\end{thm}

Note that 
$H_F$ is precisely the semigroup generated by $\{(u-v)\setsep ~u \in A,~ v \in \Ext(F)\}$.

We split the proof into a number of simple lemmas:
\begin{lem}\label{lem:H_F_subset_E_F_A}
	For any face  $F$  of $\conv(A)$ we have
	\[ \conv(H_F) = \Env_F.\] 
\end{lem}
\begin{proof}
	For any finite set $B \subseteq \mathbb{R}^d$, 
	\[ \conv\left(\left\{ \sum_{v \in B} a_v v \setsep a_v \in \mathbb{Z}_+\right\}\right) = \left\{ \sum_{v \in B} \alpha_v v \setsep \alpha_v \in [0,\+\infty) \right\}.\]
	The result follows immediately  from the expression \eqref{eq:H_F_module} for $H_F$ and  the expression \eqref{eq:Env_alt_exp} for $\Env_F$.
\end{proof}
\begin{lem}\label{lem:H_F_is_horoball}
	Let $F$ be a face of $\conv(A)$. Then the set $H_F$ given by \eqref{eq:H_F_module} is an $A$-horoball.	
\end{lem}
\begin{proof}

%	\end{equation}
	Let $m = | \Ext(F)|$. For any $j \in \N$ let
	\[ B_j := \sum_{w \in \Ext(F)} j(A-w).\]
	Then
	\[B_j = (-j\sum_{w \in \Ext(F)}w) + jm A\]
	is a translate of $jmA$.
	Since $B_{j} \subseteq B_{j+1}$ for every $j$, 
	\[
     \lim_{j \to \infty} B_j =\bigcup_{j=1}^\infty B_j
     =\bigcup_{j=1}^\infty \sum_{w \in \Ext(F)} j(A-w)=H_F.\] This
   shows $H_F$ is an increasing union of translates of iterated sums
   of $A$. It is clearly non-empty and by Lemma
   \ref{lem:H_F_subset_E_F_A} it is contained in a half-space so
   $H_F \ne \Z^d$. This proves $H_F$ is indeed a horoball.
\end{proof}

We will only need the following lemma later but we prove it here since
it is very close in spirit to the previous one.
\begin{lem}\label{lem:Env_cap_Zd_contained_H_F}
	For any face $F$ of $\conv(A)$ there exists $v \in \Z^d$ so that 
	$$ \Env_F \cap \Z^d \subseteq v + H_F.$$
\end{lem}
\begin{proof}
	As in the proof of \Cref{lem:H_F_is_horoball}, for every $N \in \mathbb{N}$ we have
	\begin{equation}	\label{eq:H_F_alt}
	H_F = \bigcup_{j=N}^\infty  \left((-j\sum_{w \in \Ext(F)}w) + jm A\right),
	\end{equation}
	where $m=| \Ext(F)|$.
	By \Cref{prop:iterated_mikowski_sums_convex}  there exists $N \in
   \N$ so that for any $j \geq N$  
	\[(j-N)\conv(A) \cap \Z^d \subseteq jA.\]
	From this it follows that for every $j \ge N$% and $w \in \Ext(F)$  
	\[(-j\sum_{w \in \Ext(F)}w)+(mj-N)\conv(A) \cap \Z^d \subseteq jmA+(-j\sum_{w \in \Ext(F)}w).\]
	Taking the union over $j \ge N$, using \eqref{eq:H_F_alt}, it follows that 
	$$\bigcup_{j =N}^\infty \left( (-j\sum_{w \in \Ext(F)}w) +(mj-N)\conv(A) \cap \Z^d \right)\subseteq  H_F.$$
	For each $j$ the set
	$ W_j = (-j\sum_{w \in \Ext(F)}w) +(mj-N)\conv(A)$
	consists precisely of vectors $u \in \R^d$ of the form $u=  \sum_{v \in \Ext(A)}a_v v$
	so that   $\sum_{v \in \Ext(A)}a_v = -N$, $a_v \ge 0$ for $v \in \Ext(A) \setminus \Ext(F)$, and $a_w \ge -j$ for $w \in \Ext(F)$.
	Also, $u \in \Env_F$ if and only if it is of the form $u = \sum_{v \in \Ext(A)}a_v v$ with 
	$\sum_{v \in \Ext(A)}a_v =0$ and $a_v \ge 0$ for any $v \in \Ext(A)\setminus \Ext(F)$. It follows that for $v_0 \in \Ext(F)$,  $\bigcup_{j} W_j = \Env_F - Nv_0$. Thus
	\[ \Env_F \cap \Z^d 	\subseteq Nv_0 + H_F. \qedhere\]
\end{proof}

Here is a simple criterion for a semigroup in $\mathbb{Z}^d$ to be a group:
\begin{lem}\label{lem:semigroup_ZD_group}
	Let $S \subseteq \Z^d$ be a semigroup. Then $S$ is a group if and only if the convex hull of $S$ is equal to the linear span of $S$.
	Equivalently, a semigroup $S \subseteq \Z^d$ is a group if and only if it is not contained in a proper half-space of its linear span.
\end{lem}
\begin{proof}
	Any subgroup of $\Z^d$ is a lattice in its linear span, and thus its convex hull  is equal to its linear span.
	Conversely, suppose $S \subseteq \Z^d$ is a  semigroup and that the convex hull of $S$ is equal to the linear span of $S$.
	Then the rational convex hull of $S$ is equal to the rational span of $S$. Now fix $a \in S$. Then in particular $-a$ is in the rational convex hull of $S$, so there exists $q_1,\ldots,q_n \in \mathbb{Q} \cap [0,\infty)$ and
	$v_1,\ldots ,v_m \in S$ so that
	\[ -a = \sum_{i=1}^mq_i v_i.\]
	Multiplying by the common denominator of the $q_i$'s, we see that for some positive integer $N$ and positive integers $n_1,\ldots,n_m$ we have $-Na=\sum_{i=1}^m n_i v_i$. This shows that $-Na$ is in the semigroup. Thus $-a= -Na + (N-1)a$ is also in the semigroup $S$. So the semigroup $S$ is closed under inverses, thus it is a group.
\end{proof}
For a set $L \subseteq \Z^d$ we denote the stabilizer  of $L$ by

\[\stab(L) := \left\{v \in \Z^d \setsep L+v = L \right\}.\]

The set $\stab(L)$ is always a subgroup of $\Z^d$.

Also, for a face $F$ of $\conv(A)$ let $L_F$ denote the linear span of $\{ u-v \setsep u,v \in F \cap A\}$.
Equivalently,
\[ L_F = \bigcup_{F' \in M_F}\left\{ v \in \mathbb{R}^d \setsep \langle \ell_{F'} , v \rangle =0 \right\}.\]

The following lemma identifies the stabilizers of horoballs in $\Z^d$.
\begin{lem}\label{lem:H_F_stab}
	For any face  $F$  of $\conv(A)$, 
	\[ \stab(H_F) = H_F \cap L_F.\]
	%the stabilizer $\stab(H_F)$ is equal to the intersection of $H_F$ with 
	% the linear span of $\{ u-v \setsep u,v \in F \cap A\}$.
	Equivalently, $\stab(H_F)$ is equal to the group generated by $\{
   u-v \setsep u,v \in F \cap A\}$.
   In particular $\stab(H_F)$ is a finite index subgroup of $L_F \cap \Z^d$.
\end{lem}
\begin{proof}
	Set $G_F := H_F \cap L_F$.
	From \eqref{eq:H_F_module} and the fact that $F$ is a face of the convex hull of $A$,
	\[ G_F = \sum_{v \in \Ext(A) \cap F} \bigcup_{j=0}^\infty j( (A
     \cap F) -v).\] In particular, $G_F$ is a semigroup.  Also since
   any vector of the form $v -w$ with $v,w \in \Ext(A \cap F)$ is in
   $G_F$, the convex hull of $G_F$ is equal to $L_F$, which contains the
   linear span of $G_F$.  Thus by \Cref{lem:semigroup_ZD_group} $G_F$
   is a group and therefore contains any vector of the form $u-v$ with
   $u,v \in A \cap F$.

   Now since $0 \in H_F$ , for any
   $v \in \stab(H_F)$, $v \in v +H_F = H_F$. If
   $v \in H_F \setminus L_F$ then the convex hull of $H_F +v$ is
   contained in the interior of the convex hull of $H_F$, so
   $v \not\in \stab(H_F)$. It follows that
   $\stab(H_F) \subseteq H_F \cap L_F$. Conversely, it follows
   directly from \eqref{eq:H_F_module} that $u-v \in H_F$ for every
   $u,v \in A \cap F$. Since $\stab(H_F)$ is a group, this completes
   the proof.
\end{proof}

The following result, which will be used in \Cref{sec:nat_ext_Z_d},
expresses the dynamics of $\varphi_A$ on horoballs $H_F$:
 \begin{lem}\label{lem:varphi_H_F}
 	For any face $F$ of $\conv(A)$ and any $v \in \Ext(F)$ we have
 	\[\varphi_A(H_F)= H_F +v.\]
 \end{lem}
\begin{proof}
  Let $F$ be a face of $\conv(A)$ and let $v \in \Ext(F)$.
  Since $v \in A$, we have $H_F+v \subseteq H_F+A=\varphi_A(H_F)$.
  On the other hand, suppose $u \in \varphi_A(H_F)= H_F +A$. Then by definition of $H_F$,
	 there exists $w,w_1,\dots,w_j \in A$ and $v_1,\ldots v_j \in \Ext(F)$ so that
	 $u = \sum_{k=1}^j (w_k-v_k) + w$.
	 It follows that $u=\sum_{k=1}^j (w_k-v_k) +(w-v) + v \in H_F +v$.
	This shows that $\varphi_A(H_F) \subseteq H_F +v$.\end{proof}
\begin{lem}\label{lem:monoid_faces}
	Let $W \subseteq \Ext(A)$ be a non-empty subset.
	If  $F$ is the minimal face of  $\conv(A)$ that contains $W$, then 
	$$ H_F = \sum_{v \in W} \bigcup_{j =0}^\infty j(A-v).$$
\end{lem}

\begin{proof}
	The point $\frac{1}{|W|}\sum_{w \in W}w$
	is contained in the relative interior of $F$.
	% Hence there is $n \in \N$ such that 	
	% \begin{align*}
	% \frac{n}{(n-1)|W|}\sum_{w \in W}w - \frac{1}{(n-1)|\Ext(F)|}\sum_{v \in \Ext(F)}v
	% = \frac{1}{n-1} \sum_{v \in \Ext(F)} \alpha_v v
	% \end{align*}
	% for some $\alpha_v \in \N$ with $\sum_{v \in \Ext(F)}
   % \alpha_v=n-1$. {\color{blue}Why are the $\alpha_v$ positive?}
	% Multiplying by the common denominator we get
   Therefore we can represent it as the rational convex
   combination of the extremal points of $F$ with all coefficients
   positive.
   Multiplying by the common denominator we obtain $m \in \N$ and  positive integer
   coefficients
   $(\beta_v)_{v \in \Ext(F)}$ such that 
   \begin{align*}
   m \sum_{w \in W} w &= \sum_{v \in \Ext(F)} \beta_v v
	\end{align*}
	and 
	\begin{align*}
	\sum_{v \in \Ext(F)} \beta_v=m|W|.
	\end{align*}
	
	Note that for any $w \in W$, $\bigcup_{j=0}^\infty j(A-w)=\bigcup_{j=0}^\infty jm(A-w)$. 
	Therefore 
	\begin{align*}
	\sum_{w \in W} \bigcup_{j =0}^\infty j(A-w) &= \bigcup_{j=0}^\infty jm|W|A - jm\sum_{w \in W} w \\
	& =\bigcup_{j=0}^\infty jm|W|A - j\sum_{v \in \Ext(F)} \beta_v v \\
	&= \sum_{v \in \Ext(F)} \bigcup_{j=0}^\infty j\beta_v (A- v)\\
	&=\sum_{v \in \Ext(F)} \bigcup_{j=0}^\infty j (A- v) \\
	&=H_F. \qedhere
	\end{align*}
\end{proof}

%\begin{lem}
%	For each $v$ in a finite set $W$ let $(L_{k,v})_{k \in \N}$ be an increasing sequence
%	of sets in $\Z^d$. Then
%	\begin{align*}
%	\bigcup_{k=1}^\infty \sum_{v \in W} L_{k,v} =  \sum_{v \in W} \bigcup_{k=1}^\infty L_{k,v}.
%	\end{align*}  
%\end{lem}
\begin{proof}[Proof of Theorem \ref{thm:ZD_horoballs_are_faces_of_convex_hull}]
	
	We first show that any $A$-horoball is of the form $v + H_F$ for some $v \in \Z^d$ and a face $F$ of $\conv(A)$.
	By \Cref{lem:monoid_faces} it suffices to show that any $A$-horoball is  a translate of a set of the form 
	$\sum_{w \in W}\bigcup_{j=1}^\infty j(A-w)$ for some $W \subseteq \Ext(V)$. 
	Let $L$ be a horoball containing $0$. There is a sequence of sets
	of the form $-b_{k}+n_kA$ with increasing $n_k$ and $b_{k} \in n_kA$
	converging to $L$. By the Shapley-Folkman Lemma as in the proof of \Cref{prop:iterated_mikowski_sums_convex}, 
	passing to  a subsequence we can ensure that there is $r \in \N$ and 
	$s \in rA$ such that $b_{k} \in s + (n_k-r)\Ext(A)$. Hence
	we can find tuples $(\alpha_{v,k})_{v \in \Ext(A)}$ of non-negative integers such
	that $b_{k}=s+\sum_{v \in \Ext(A)} \alpha_{v,k} v$ and $\sum_{v \in \Ext(A)} \alpha_{k,v}=n_k-r$.
	%Using \Cref{lem:strictly-incr-coord-wise} and 
	Again passing to a subsequence we may assume that for every
   $v \in \Ext(A)$  the sequence $(\alpha_{v,k})_{k \in \N}$ is
   non-decreasing and that therefore the limit
   $\alpha_v := \lim_{k \to \infty}\alpha_{v,k} \in \N_0 \cup
   \{+\infty\}$ exists. Let
   $W = \left\{v \in \Ext(A):~ \alpha_v = + \infty \right\}$ be the
   set of vertices $v$ for which $(\alpha_{v,k})_{k \in \N}$ tends to
   $+\infty$.
	%Let $F$ be the minimal face of $\conv(A)$ with the property that $W$ is contained in $\aff(F)$. 
	We claim that 
	$L$ is a translate of 	$\sum_{w \in W}\bigcup_{j=1}^\infty j(A-w)$.
	
	First of all
	\begin{align*}
	-b_k+n_k A &= -s - (\sum_{v \in W} \alpha_{v,k} v) + n_k A \\
	&=-s +rA + \sum_{v \in W} \alpha_{v,k}(A-v)
	\end{align*}
	and hence $L$ is the increasing union of the sets
	$-b_k+n_k A$.
	Now
	\begin{align*}
	L=\bigcup_{k \in \N} -b_k+n_k A
	&= -s+rA+ \bigcup_{k \in \N} \sum_{v \in W} \alpha_{v,k}(A-v) \\
	&= -s+rA + \sum_{v \in W} \bigcup_{k \in \N} \alpha_{v,k}(A-v) \\
	&= -s+rA + \sum_{v \in W} \bigcup_{k \in \N} k(A-v).
	\end{align*}
	Now choose $w \in W$. Then $rA = r(A-w) + rw$. Since $(A-w)$ is contained in the semigroup  $\sum_{v \in W} \bigcup_{k \in \N} k(A-v)$ and contains $0$ we have:
	\[r(A-w)  +\sum_{v \in W} \bigcup_{k \in \N} k(A-v)= \sum_{v \in W} \bigcup_{k \in \N} k(A-v).\]
	It follows that
	\[L = -s + rw +  \sum_{v \in W} \bigcup_{k \in \N} k(A-v).\]
	This completes the proof that any $A$-horoball is a translate of
   some $H_F$. Let $v + H_F$ be an $A$-horoball for some face $F$ of
   $\conv(A)$ and $v \in \Z^d$ such that $0 \in v+H_F$ and so that
   $\conv(v+H_F ) = \Env_F$.
	By \Cref{lem:H_F_subset_E_F_A}, $\conv(v+H_F)= v+ \Env_F$. Thus, $v+\Env_F=\Env_F$ so $v \in L_F$. But $0 \in v +H_F$ so $-v \in H_F$. By \Cref{lem:H_F_stab} it follows that $v+H_F=H_F$.
	
\end{proof}

\section{Horoballs in $\Z^d$ as a topological space and a topological dynamical system}\label{sec:Z_d_top_dynam}
In this section we make a brief digression from the study of
$\varphi_A$ to study the space $\overline{\Hor(\Z^d,A)}$ as a compact totally disconnected topological space  and as $\Z^d$-topological
dynamical system. 

For $u \in \R^d$ and $R >0$ we let $B_R(u)$ denote the Euclidean ball of radius $R$ centered at $u$.

%The following Lemma is needed to conclude the Cantor-Bendixon structure of $\Hor(\Z^d,A)$.

%For faces $F,F'$ of $\conv(A)$ with $F \subseteq F'$. Let $M_{F,F'}$
%denote the intersection of $H_{F'}$ with the hyperplane $L_{F}$ spanned by $\{v-w \setsep v,w \in F\}$.
%It follows that
%\[M_{F,F'}= \bigcup_{j=1}^\infty \sum_{v \in  F}j (\Ext(F') -v).\]

\begin{lem}\label{lem:conv_H_F_local}
There exists $R>0$ such that for any face $F$ of $\conv(A)$, 
if $u\in \conv(H_F)$ then $u \in\conv( H_F \cap B_{R}(u))$.
%$v \in \mathbb{R}^d$ and $r>0$
%\[ \conv(H_F) \cap B_r(v)  = \conv( H_F \cap B_{r+R}(v)) \cap B_r(v).\]
\end{lem}
\begin{proof}
Choose $u \in \conv(H_F)$. Then $u$ is of the form 
\[ u= \sum_{v \in \Ext(F)}\sum_{w \in A} \alpha_{v,w}(w-v),\]
with $\alpha_{v,w} \in [0,\infty)$.
For $\alpha \in \mathbb{R}$ let $[\alpha]_0 = \lfloor \alpha \rfloor$ and $[\alpha]_1 =\lceil \alpha \rceil$.
Then $u$ is in the convex hull of the set
\[
A_u = \left\{\sum_{v \in \Ext(F)}\sum_{ w \in A} [\alpha_{v,w}]_{f(v,w)}(w-v)  \setsep f \in \{0,1\}^{\Ext(F) \times A} \right\}. 
\] 
Let 
\[ A'= \left\{ \sum_{v \in \Ext(F)}\sum_{ w \in A} b_{v,w}(w-v)  \setsep b \in [-1,1]^{\Ext(F) \times A}\right\}.\]
Then $A_u \subseteq H_F \cap (u + A')$.
Choose $R>0$ sufficiently large so that $A'$ is contained in $B_R(0)$. Then  $u \in\conv( H_F \cap B_{R}(u))$. 
\end{proof}

Recall that for a face $F$ of $\conv(A)$, $M_F$ has been defined in \Cref{defn:M_F}.

\begin{lem}\label{lem:conv_F_1_to_conv_F_2}
Suppose $F_1, F_2$ are faces of $\conv(A)$,
and that
$(v_i)_{i \in \N}$ is a sequence of points in $\Z^d$. Then 
$v_n+ \Env_{F_1} \to \Env_{F_2}$ as $n \to \infty$ (with respect to the product topology on $\{0,1\}^{\R^d}$) if and only if $F_1 \subseteq F_2$ and for every $F' \in M_{F_1}$ we have
\begin{equation}\label{eq:ell_F_dot_prod_lim}
\lim_{n \to \infty} \langle \ell_{F'},v_n\rangle = \begin{cases}
-\infty & \text{ if }F'  \in M_{F_1} \setminus M_{F_2}\\
0  & \text{ if }F' \in M_{F_2}.
\end{cases}
\end{equation}
\end{lem}
\begin{proof}
Note that $F_1 \subseteq F_2$ implies $M_{F_2} \subseteq M_{F_1}$.
If $\langle \ell_{F'},v_n\rangle < -R$ for every $F' \in  M_{F_1} \setminus M_{F_2}$ and
$\langle \ell_{F'},v_n\rangle =0$  for every $(d-1)$-dimensional face $F'$ of $\conv(A)$ which contains $F_2$, then 
\[(v_n + \Env_{F_1}) \cap B_R(0) = \Env_{F_2} \cap B_R(0).\]
This proves that the conditions $F_1 \subseteq F_2$ together with \eqref{eq:ell_F_dot_prod_lim} imply that $v_n+ \Env_{F_1} \to \Env_{F_2}$ as $n \to \infty$.

Conversely, suppose that $v_n+ \Env_{F_1} \to \Env_{F_2}$ as
$n \to \infty$.  Since $0 \in \Env_{F_2}$, it follows that
$0 \in v_n + \Env_{F_1}$ and thus
$-v_n \in \Env_{F_1}$ for all sufficiently large $n$. This implies
that $\Env_{F_1} \subseteq v_n + \Env_{F_1}$ for all sufficiently
large $n$. Taking $n \to \infty$, we deduce that
$\Env_{F_1} \subseteq \Env_{F_2}$, which implies $F_1 \subseteq
F_2$. Let $F' \in M_{F_1}$. Define points
\begin{align*}
  m_{F_1}&:=\frac{1}{|\Ext(F_1)|} \sum_{w \in
  \Ext(F_1)} w \in F_1 \subseteq F',\\
  u&:=\sum_{v \in \Ext(F')}(v -m_{F_1}).
\end{align*}

Suppose $F' \in M_{F_2}$.  Let $\epsilon>0$.  We have
$\langle u, \ell_{F'} \rangle = 0$.  For
$F'' \in M_{F_1} \setminus \{F'\}$ we have
$\langle u, \ell_{F''} \rangle > 0$.  We thus can choose $\alpha>0$
large enough such that for all these $F''$ we have
$\langle \alpha u - \epsilon \ell_{F'}, \ell_{F''}\rangle >0$.  Since
$\langle \alpha u - \epsilon \ell_{F'},\ell_{F'}\rangle =
-\epsilon<0$, our point $\alpha u - \epsilon \ell_{F'}$ is not
contained in $\Env_{F_2}$, and hence also not in $v_n +\Env_{F_1}$ for
sufficiently large $n \in \N$.  As shown above, we may also assume
that $-v_n \in \Env_{F_1}$.  Therefore there must be $F'' \in M_{F_1}$
with
$\langle \alpha u - \epsilon \ell_{F'}-v_n, \ell_{F''}\rangle <0$.
Since $\langle -v_n, \ell_{F''} \rangle \geq 0$ this implies
$\langle \alpha u - \epsilon \ell_{F'}, \ell_{F''}\rangle <0$.  By our
choice of $\alpha$ this can only happen for $F''=F'$. But then
$\langle \alpha u
- \epsilon \ell_{F'}-v_n, \ell_{F'}\rangle<0$, thus
$-\epsilon = \langle \alpha u- \epsilon \ell_{F'},\ell_{F'}\rangle <\langle
v_n,\ell_{F'}\rangle \leq 0$. Hence
$\lim_{n \to \infty} \langle v_n,\ell_{F'}\rangle=0$.

Suppose
$F' \in M_{F_1} \setminus M_{F_2}$. For every $F'' \in M_{F_2}$ we
have
$\langle u,\ell_{F''}\rangle>0$. It follows that for every $R >0$ ,
there exists $Q>0$ such that $Qu-R \ell_{F'} \in v_n + \Env_{F_1}$
for all sufficiently large $n$. So there is $w_n \in \Env_{F_1}$
such that $v_n=Qu-R\ell_{F'}-w_n$. This implies that
\[\langle \ell_{F'},v_n \rangle= \langle \ell_{F'},Qu-R\ell_{F'}-w_n\rangle  = -R -\langle
  \ell_{F'}, w_n\rangle\leq -R \]
and thus
\[ \lim_{n \to \infty} \langle \ell_{F'},v_n\rangle  =
  -\infty,\]
completing the proof.
\end{proof}

\begin{lem}\label{lem:H_F_to_emptyset}
Let $F$ be a face of $\conv(A)$ and  let
$(v_n)_{n \in \N}$ be a sequence of points in $\Z^d$.
If 
\[ \lim_{n \to \infty} \max_{F' \in M_F}  \langle \ell_{F'},v_n\rangle =+\infty, \]
then $(v_n + H_F) \to \emptyset$ as $n \to \infty$.
\end{lem}
\begin{proof}
  For any $R >0$ there exists $Q >0$ such that
  $\max_{F' \in M_F} \langle \ell_{F'},v\rangle < Q$ for every
  $v \in B_R(0)$. It follows that whenever
  $\max_{F' \in M_F} \langle \ell_{F'},v_n\rangle \geq Q$ we have
  $(v_n + H_F) \cap B_R(0) = \emptyset$.  Thus the condition
\[ \lim_{n \to \infty} \max_{F' \in M_F}  \langle \ell_{F'},v_n\rangle =+\infty.\]
implies that  $(v_n + H_F) \to \emptyset$ as $n \to \infty$.
\end{proof}

\begin{lem}\label{lem:v_n_H_F_lims}
  Suppose $F_1,F_2$ are faces of $\conv(A)$ and that
  $(v_i)_{i \in \N}$ is a sequence of points in $\Z^d$.  For ease of
  notation we allow $F_2 = \conv(A)$ (considered as a
  ``$d$-dimensional face''), in which case we denote $H_{F_2}= \Z^d$.
  Then $(v_n+ H_{F_1}) \to H_{F_2}$ as $n \to \infty$ if and only if
  $\left( v_n+ \conv(H_{F_1}) \right)\to \conv(H_{F_2})$ as $n \to \infty$ and there
  exists $N \in \mathbb{N}$ so that $v_n \in \stab(H_{F_2})$ for all
  $n >N$.
\end{lem}
\begin{proof}

  Suppose $(v_n + H_{F_1}) \to H_{F_2}$.  Let $R>0$ be as in
  \Cref{lem:conv_H_F_local}.  Recall that $\Env_{F_1}= \conv(H_{F_1})$
  and $\Env_{F_2}= \conv(H_{F_2})$.  Let $u \in \R^n$. Since
  $(v_n + H_{F_1}) \to H_{F_2}$, it follows that there exists $N \in \N$ such
  that $H_{F_2} \cap B_R(u) =(v_n+H_{F_1}) \cap B_R(u)$ for all
  $n >N$. Thus, for all $n >N$,
  $\conv(H_{F_2} \cap B_R(u)) = \conv((v_n+H_{F_1}) \cap B_R(u))$. By
  \Cref{lem:conv_H_F_local},
  $u \in \conv(H_{F_2})$ iff $u \in \conv(H_{F_2} \cap
  B_R(u))=\conv((H_{F_1}+v_n) \cap B_R(u)) = v_n + \conv(H_{F_1} \cap
  B_R(u-v_n))$ and this is the case iff $u \in v_n + \conv(H_{F_1})$.
  Hence $(v_n + \conv(H_{F_1})) \to \conv(H_{F_2})$.

By \Cref{lem:conv_F_1_to_conv_F_2} this implies that $v_n$ is in the
linear span of $\{ (v-w) \setsep v,w \in \Ext(F_2)\}$ for sufficiently
large $n$ and that $F_1 \subseteq F_2$. Since $v_n
+ H_{F_1} \to H_{F_2}$, it follows that $0 \in v_n + H_{F_1}$ for
large $n$. So $-v_n \in H_{F_1}$. By \Cref{lem:H_F_stab}, the
intersection of  the linear span of $\{ (v-w) \setsep v,w \in
\Ext(F_2)\}$ with $H_{F_1}$ is contained in $\stab(H_{F_2})$. We conclude that there
exists $N \in \mathbb{N}$ so that $v_n \in \stab(H_{F_2})$ for all
$n >N$.

Conversely, suppose that $v_n+ \conv(H_{F_1}) \to \conv(H_{F_2})$ as
$n \to \infty$ and there exists $N \in \mathbb{N}$ so that
$v_n \in \stab(H_{F_2})$ for all $n >N$.   
 By compactness, it suffices
to show that for any converging subsequence
$(v_{n_k} + H_{F_1})_{k \in \N}$ , we have
$(v_{n_k}+ H_{F_1}) \to H_{F_2}$ as $k \to \infty$.  Let
$(v_{n_k} + H_{F_1})_{k \in \N}$ be a converging subsequence, so
$(v_{n_k} + H_{F_1}) \to W$ for some $W \subseteq \Z^d$. 
By \Cref{lem:conv_F_1_to_conv_F_2} 
 $v_n+ \conv(H_{F_1}) \to \conv(H_{F_2})$ implies $F_1 \subseteq F_2$.
Since $v_n \in  \stab(H_{F_2}) \subseteq H_{F_2}$,  and $H_{F_2}$ is a semigroup, it follows that
$v_n + H_{F_1} \subseteq H_{F_2}$ for all $n>N$.  This implies that $W \subseteq H_{F_2}$.
 By
\Cref{thm:ZD_horoballs_are_faces_of_convex_hull} the limit $W$ must
either be of the form $v+H_F$ for some face $F$ of $\conv (A)$, or
$\Z^d$ or $\emptyset$.  But the first part of the proof shows that
$\conv(W) = \conv(H_{F_2})$. It follows that $W = v + H_{F_2}$ for
some $v \in \Z^d$ in the linear span of
$\{ (u-w) \setsep u,w \in \Ext(F_2)\}$. 
We have already concluded that $v+ H_{F_2}=W \subseteq H_{F_2}$.
Using \Cref{lem:H_F_stab}, the conditions $v+ H_{F_2} \subseteq  H_{F_2}$ and $v + \conv(H_{F_2}) = \conv(H_{F_2})$ together imply that 
 $v \in \stab(H_{F_2})$
 and
$(v_{n_k} + H_{F_1}) \to H_{F_2}$.
\end{proof}

\begin{lem}\label{lem:H_F_orbit_limits}
Let $F$ be a face of $\conv(A)$ and let $(v_i)_{i \in \N}$ be a sequence of points in $\Z^d$.
Then:
\begin{enumerate}[(1)]
\item Suppose $\lim_{n \to \infty} \max_{F' \in M_F}  \langle
  \ell_{F'},v_n\rangle =+\infty$.
  Then
  $(v_n + H_F) \to \emptyset$ as $n \to \infty$.
\item Suppose that
  $\lim_{n \to \infty} \langle \ell_{F'},v_n\rangle =-\infty$ for
  every $(d-1)$-dimensional face $F'$ of $\conv(A)$ containing
  $F$. Then $(v_n + H_F) \to \Z^d$ as $n \to \infty$.
\item Otherwise, suppose that the limit
  $\lim_{n \to \infty} \langle \ell_{F'},v_n\rangle \in
  [-\infty,+\infty)$ exists for every $F' \in M_F$ and let $F''$
  denote the face of $\conv(A)$ which is the intersection of all
  $F' \in M_F$ such that
  $\lim_{n \to \infty} \langle \ell_{F'},v_n\rangle$ is
  finite. Furthermore, suppose all but finitely many elements of
  $(v_n)_{n \in \N}$ belong to a fixed coset of $\stab(H_{F''})$. Then
  there exists $v \in \Z^d$ so that $v_n + H_F \to v+H_{F''}$ as
  $n \to \infty$.
\item In all other cases, the sequence $v_n+H_F$ does not converge.
\end{enumerate}
\end{lem}
\begin{proof}
%For a $(d-1)$-dimensional face $F'$ of $\conv(A)$ and $M \in \mathbb{R}$ let
%\[K_{F',M} = \left\{v \in \mathbb{R}\setsep   \langle \ell_{F'},v_n\rangle \le M\right\} .\]
\begin{enumerate}
\item[(1)]  If $\lim_{n \to \infty} \max_{F' \in M_F}  \langle \ell_{F'},v_n\rangle =+\infty$, then $(v_n + H_F) \to \emptyset$ by \Cref{lem:H_F_to_emptyset}.
\item[(2)+(3)] 
Note that $(2)$ is essentially a particular case of $(3)$ if we agree that $F''=\conv(A)$
 when 
  $\lim_{n \to \infty} \langle \ell_{F'},v_n\rangle=-\infty$ for all
  $F' \in M_F$.  By \Cref{lem:conv_F_1_to_conv_F_2}, $v_n + \conv(H_F) \to v + \conv(H_{F''})$. Using the fact that $v_n -v \in \stab(H_{F''})$ for all $n \ge N$,  by \Cref{lem:v_n_H_F_lims}, it follows that $v_n + H_F \to v+H_{F''}$.

\item[(4)]  Let $(v_i)_{i \in \N}$ be a sequence of points in $\Z^d$,
  and suppose that the sequence $(v_n +H_F)_{n \in \N}$ converges.
  By compactness this sequence converges if and only if all of its
  converging subsequences have the same limit.  Then by
  \Cref{thm:ZD_horoballs_are_faces_of_convex_hull} the possible limit points are
  $\emptyset$, $\Z^d$ and $v+H_{F'}$ for faces $F'$ of $\conv(A)$ and
  $v \in \Z^d$. If
  $\lim_{n \to \infty} \max_{F' \in M_F} \langle \ell_{F'},v_n\rangle
  =+\infty$, we are in case $(1)$. Otherwise we know from
  \Cref{lem:v_n_H_F_lims} that $(v_n + \conv(H_{F_1})) \to
  (v + \conv(H_{F_2}))$ and that $v_n \in (v+\stab(H_{F_2}))$ for all
  sufficiently large $n \in \N$. By \Cref{lem:conv_F_1_to_conv_F_2} we
  can conclude that if $(v_n + H_F)_{n \in \N}$
converges, we are either in case $(1)$, case $(2)$ or case $(3)$. \qedhere
\end{enumerate}
\end{proof}

For a topological space $X$ let $X'$ denote the set of accumulation points of $X$.
The \emph{$k$-th Cantor-Bendixson derivative} is inductively defined as $X^{(0)}=X$ and 
\[ X^{(k+1)} = (X^{(k)})'.\]
Later in \Cref{sec:eventual-image-zd} we will also need the
$\alpha$-th Cantor-Bendixson derivative for an ordinal number
$\alpha$. For successor ordinals the same definition as above holds
and for limit ordinals $\alpha$ it is defined as
\[ X^{(\alpha)} = \bigcap_{\beta < \alpha} X^{(\beta)}.\]
A topological space $X$ has \emph{Cantor–Bendixson rank $\alpha$} iff
$\alpha$ is the smallest ordinal with $X^{(\alpha)} = X^{(\alpha+1)}$.

Using Theorem \ref{thm:ZD_horoballs_are_faces_of_convex_hull} and  \Cref{lem:H_F_orbit_limits} we can summarize the topological structure of $\Hor(\Z^d,A)$ as follows:
\begin{thm}\label{thm:Hor_topological_structure}
	Let $A \subseteq \Z^d$ be a finite positively generating set.
	Then the closure of $\Hor(\Z^d,A)$ in $\PP(\Z^d)$ is a countable compact subset of $\PP(\Z^d)$ having Cantor–Bendixson rank $d+1$.
	Furthermore,
	$$\overline{\Hor(\Z^d,A)}^{(d)}= \left\{\emptyset, \Z^d \right\},$$
	 and for $0 \le k < d$ the isolated points of
	$\overline{\Hor(\Z^d,A)}^{(k)}$ are precisely \[\{v + H_F\setsep F
   \text{ is a  $k$-dimensional face of  $\conv(A)$ and } v \in \Z^d\}.\]
\end{thm}
\begin{proof}
  Let $F,F'$ be  faces of $\conv(A)$. 
By \Cref{lem:H_F_orbit_limits}, the  horoball $H_{F'}$ is a limit point of
the orbit of $H_F$ under translations if
  and only if $F \subseteq F'$.
This show that $v+H_F$ is an isolated point of $\overline{\Hor(\Z^d,A)}^{(k)}$  if and only if $F$ is a $k$-dimensional face of  $\conv(A)$.
\end{proof}

\begin{cor}\label{cor:Hor_Z_d_top_invariant}
  Suppose that $d_1,d_2\in \N$, $A_i \subseteq \Z^{d_i}$ are finite
  positively generating sets for $i=1,2$. Then
  $\overline{\Hor(\Z^{d_1},A_1)}$ is homeomorphic to
  $\overline{\Hor(\Z^{d_2},A_2)}$ if and only if $d_1=d_2$.
 \end{cor}
 \begin{proof}
 \Cref{lem:H_F_orbit_limits} shows that for any finite positively generating set $A$ of $\Z^d$, if $X_A= \overline{\Hor(\Z^{d},A)}$ then $X_A^{(d)}=\{\emptyset,\Z^d\}$ consists of $2$ points.
It follows from a theorem of Mazurkiewicz and  Sierpi{\'n}ski \cite{mazurkiewicz1920contribution} that any
such countable compact metrizable topological space of
Cantor-Bendixson rank $d+1$ is homeomorphic to the ordinal $2\omega^d +1$, with the order topology.
 \end{proof}

\begin{prop}\label{prop:conjugacy_of_Hor_ZD_A}
  Suppose $A_1,A_2 \subseteq \Z^{d}$ are finite positively generating
  sets.  Then the $\Z^d$ actions by translations on
  $\overline{\Hor(\Z^{d},A_1)}$ and $\overline{\Hor(\Z^{d},A_2)}$ are
  topologically conjugate if and only if there is a bijection $\Phi$
  between the faces  of $\conv(A_1)$ and the faces of $\conv(A_2)$
  such that for every face $F$ of $\conv(A_1)$, the following two conditions
  hold:
\begin{enumerate}
\item
  \[ \sum_{u,v \in F \cap A_1} \Z (u-v) = \sum_{u',v' \in \Phi(F)
      \cap A_2} \Z (u' - v').\] In other words, the group generated by
  differences of elements in $F \cap A_1$ is equal to the group
  generated by differences of elements in $\Phi(F) \cap A_2$.
%The group generated $\{ u-v\setsep u,v \in F \cap A_1\}$ is equal to the group generated by  $\{ u-v\setsep u,v \in \Phi(F) \cap A_2\}$.
\item
\[\Env_{F,A_1}=\Env_{\Phi(F),A_2}.\]
\end{enumerate}
\end{prop}
\begin{proof}
  By \Cref{thm:Hor_topological_structure}, the orbits of the $\Z^d$
  action on $\overline{\Hor(\Z^d,A_i)}$ are the two fixed points
  $\emptyset, \Z^d$ together with the orbits of $H_F$ where $F$ runs
  over the faces of $\conv(A_i)$.  
Suppose 
%This shows that an isomorphism
$\phi:\overline{\Hor(\Z^{d},A_1)} \to \overline{\Hor(\Z^{d},A_2)}$ is
a topological conjugacy of the $\Z^d$ actions.  By
\Cref{thm:Hor_topological_structure}, the $\Z^d$-orbits of $\Hor(\Z^{d},A_i)$
which are isolated points in  $\overline{\Hor(\Z^{d},A)}^{(k)}$  (with $0 \le k <d$) are in bijection with
the $k$-dimensional faces of $\conv(A_i)$.  Thus, for each
$k$-dimensional face $F$ of $\conv(A_1)$ there exist
$v_{F} \in \Z^d$ and a $k$-dimensional face $\Phi(F)$ of $\conv(F_2)$
such that $\phi(H_F) = v_F + H_{\Phi(F)}$.  Clearly $\Phi$ is a
bijection between the faces of $\conv(A_1)$ and the faces of
$\conv(A_2)$.  Furthermore, by \Cref{lem:H_F_orbit_limits} for faces
$F,F'$ of $\conv(A_1)$, $H_{F'}$ is an accumulation point of the orbit of $H_F$
if and only if $F \subseteq F'$. Thus the bijection $\Phi$ respects
the incidence relations between the faces of the polytopes $\conv(A_1)$ and
$\conv(A_2)$. In other words, the polytopes $\conv(A_1)$ and
$\conv(A_2)$ are \emph{combinatorially isomorphic}.  Since $\phi$ is a
topological conjugacy, we must have $\stab(H_F)= \stab(\phi(H_F))$.
By \Cref{lem:H_F_stab}, the stabilizer of $H_F$ is equal to the group
generated by differences of elements in $F \cap A_1$.  This shows that
the group generated by differences of elements in $F \cap A_1$ is
equal to the group generated by differences of elements in
$\Phi(F) \cap A_2$.  By \Cref{lem:H_F_to_emptyset},
$v_n + H_F \to \emptyset$ if and only if the distance between
$-v_n$ and $\conv(H_F)$ tends to $\infty$. This shows that
$\conv(H_F)=\conv(H_{\Phi(F)})$. By \Cref{lem:H_F_subset_E_F_A} this
implies that $\Env_{F,A_1}=\Env_{\Phi(F),A_2}$.

Conversely, suppose that there is a bijection $\Phi$ between the faces
$F$ of $\conv(A_1)$ and the faces of $\conv(A_2)$ as above.  Define
$\phi:\overline{\Hor(\Z^{d},A_1)} \to \overline{\Hor(\Z^{d},A_2)}$ by
$\phi(\emptyset)=\emptyset$, $\phi(\Z^d)=\Z^d$ and
$\phi(v+H_F) = v + H_{\Phi(F)}$ for $v \in \Z^d$ and $F$ a face of
$\conv(A_1)$.  By the assumption that the group generated by
differences of elements in $F \cap A_1$ is equal to the group
generated by differences of elements in $\Phi(F) \cap A_2$ and using
\Cref{lem:H_F_stab}, $v_1 + H_F = v_2 + H_F$ if and only if
$v_1 +H_{\Phi(F)} = v_2 + H_{\Phi(F)}$ so the map $\phi$ is well
defined and it is a bijection (here we also use
\Cref{thm:Hor_topological_structure}). It is clear that the map $\phi$
is equivariant. Combining \Cref{lem:H_F_orbit_limits} with the
assumption that $\Env_{F,A_1}=\Env_{\Phi(F),A_2}$ for every face $F$ of
$\conv(A_1)$, it follows that $(v_n + H_F) \to (v + H_{F'})$ if and only
if $(v_n + H_{\Phi(F)}) \to (v+H_{\Phi(F')})$. The analogous
statements for $\Z^d$ and $\emptyset$ hold as well. This shows that $\phi$
is a homeomorphism.
\end{proof}

\begin{figure}
  \begin{center}
    \includegraphics[scale=5]{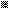}
    \hspace{1cm}
  \includegraphics[scale=5]{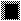}
  \end{center} 
\caption{A generating set $A$ of $\Z^2$, whose
  convex hull is a centered square, and the ball $A^3$ }.
  \label{fig:gen-and-ball}
\end{figure}

\begin{exam}
  Consider the generating set depicted on the left side of
  \Cref{fig:gen-and-ball}. Rotation by $90^{\circ}$
  gives a generating set whose space of horoballs together with
  the action by translation is conjugate to
  the original one but none of the horoballs agree.
\end{exam}

From \Cref{prop:conjugacy_of_Hor_ZD_A} it follows that there is an
algorithm to decide given two finite positively generating sets
$A_1,A_2 \subseteq \Z^d$ if the $\Z^d$ actions by translation on $\overline{\Hor(\Z^{d},A_1)}$
and $\overline{\Hor(\Z^{d},A_2)}$ are topologically conjugate, as one can reduce this to several applications of the following problems: 
\begin{itemize}
\item Given $v \in \Z^d$ and a finite set $C \subseteq \Z^d$, decide
  if $v$ is in the convex cone generated by $C$ (this is a linear programming
  problem),
\item Given $v \in \Z^d$ and a finite set $C \subseteq \Z^d$, decide if
  $v$ is in the group generated by $C$ (this amounts to solving a
  system of linear Diophantine equations).
\end{itemize}

\section{The natural extension of $\varphi_A$, $A \subseteq \Z^d$}\label{sec:nat_ext_Z_d}

Our next goal is to show that for $G=\Cayley(\Z^d,A)$ with $d \geq 2$ the
natural extension of $\varphi_A$ is perfect. This will take the remainder of this section.

\begin{lem}\label{lem:dense-set-in-evtl}
  For every finite positively generating set  $A \subseteq \Z^d$ 
  the set
\[\tilde{X}_A = \left\{ \bigcup_{i=1}^m(w_i + H_{\{v_i\}}) \setsep m\geq |\Ext(A)|,~
    w_1,\ldots,w_m \in \Z^d,~ \{v_1,\dots,v_m\}=\Ext(A)\right\}.\]
is dense in the eventual image of $\varphi_A$.
\end{lem}
\begin{proof}
  It is clear that $\tilde{X}_A$ consists of horoballunions and
  hence is a subset of the eventual image of $\varphi_A$.
  Let $R>0$. There is $n \in \N$ such that $(-nv + H_{\{v\}}) \cap
  B_R(0)=\emptyset$
  for every $v \in \Ext(A)$.
  Let $M$ be a set in the eventual image of $\varphi_A$, hence $M$ is
  the union of translates of horoballs of the form $H_{\{v\}}$ for $v
  \in \Ext(A)$. The unions of finitely many of these horoballs cover $M \cap
  B_R(0)$. Denote this union by $M_1$. Then
  $M_2:=M_1 \cup \bigcup_{v \in V} (-nv+H_{\{v\}}) \in \tilde{X}_A$.
  Furthermore $M_2 \cap B_R(0)=M \cap B_R(0)$.
  Therefore $\tilde{X}_A$ is dense in the eventual image of $\varphi_A$.
\end{proof}

\begin{figure}
  \begin{center}
  \begin{subfigure}{.48\textwidth}
    \centering
     \begin{tikzpicture}[scale=0.5]
\begin{scope}[every node/.style= {circle,inner sep=0cm, draw=black, fill= black, minimum size=0.2cm}]
\node (A1) at (0,0) {};
\node (A2) at (1,0) {};
\node (A3) at (2,0) {};
\node (A4) at (-1,1) {};
\node (A5) at (0,1) {};
\node (A6) at (1,1) {};
\node (A7) at (2,1) {};
\node (A8) at (0,2) {};
\node (A9) at (1,2) {};
\node (A10) at (2,2) {};
\node (A11) at (1,3) {};
\end{scope}
\begin{scope}[black, thin, dashed]
	\draw (A1.center) -- (A3.center)  -- (A10.center) -- (A11.center) -- (A4.center) -- cycle;	
\end{scope}
\draw[-latex] (0,-1) -- (0,4);
\draw[-latex] (-2,1) -- (3,1);
\end{tikzpicture}
\caption{A generating set $A$ of $\Z^2$.}
\label{fig:translated-cones-a}
\end{subfigure}
\begin{subfigure}{0.48\textwidth}
  \centering
  \begin{tikzpicture}[scale=0.5]
	\clip (-3,-4) rectangle (4,3);
\begin{scope}[black, opacity=0.1.5] 
	\fill (0,0) -- (-10,0) -- (0,10) -- cycle;
	\fill (0,-2) -- (0,-20) -- (-6,4) -- cycle;
	\fill (2,0) -- (20,0) -- (-4,6) -- cycle;
	\fill (2.5,-0.5) -- (8.5,5.5) -- (8,-6) -- cycle;
	\fill (0.5,-2.5) -- (-3,-6) -- (3,-5) -- cycle;
\end{scope}
	\begin{scope}[shorten <=-5cm, shorten >=-5cm, very thin]
	\draw (-2,0) -- (3,0);
	\draw (0,2) -- (0,-3);
	\draw (0,-3) -- (3,0);
	\draw (-2,0) -- (1,-3);
	\draw (0,2) -- (2,0);
\end{scope}
\begin{scope}[black, very thick]
  \draw (0,0) -- (2,0) -- (2.5,-0.5) -- (0.5,-2.5) -- (0,-2) -- cycle;
\end{scope}
\end{tikzpicture}
\caption{$-\conv(A)$ together with $\check{A}$}
\label{fig:translated-cones-b}
\end{subfigure}
\end{center}
\caption{Illustration of the proof of \Cref{lem:hausdorff-limit-of-preimages}.}
	\label{fig:translated-cones}
\end{figure}
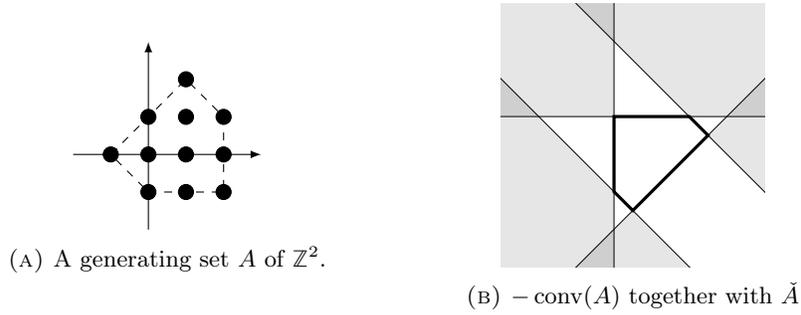
 
\begin{defn}
  For a finite positively generating set $A \subseteq \Z^d$ define $\check{A}$  by
  \[\check{A} := \bigcup_{v \in \Ext(A)}(-v + \Env_{\{v\},A}),\]
  see \Cref{fig:translated-cones} for an illustration.
\end{defn}

\begin{lem}\label{lem:hausdorff-limit-of-preimages}
  Let $A \subseteq \Z^d$ be a finite positively generating set. Let
  $m\geq |\Ext(A)|$, $w_1,\ldots,w_m \in \Z^d$ and
  $\{v_1,\dots,v_m\}=\Ext(A)$.  Set
  $M_n:=\bigcup_{i=1}^m(w_i -nv_i+ H_{\{v_i\}}) \in \tilde{X}_A$.  Then
  $M_n \in \varphi_A^{-n}(\{M_0\})$ and $\frac{1}{n} M_n$ converges to
  $\tilde{A}$ with respect to the Hausdorff metric, i.e. for all
  $\varepsilon$ and $n$ large enough we have
  $\frac{1}{n} M_n \subseteq \check{A} + B_\epsilon(0)$ and
  $\check{A} \subseteq \frac{1}{n}M_n + B_\epsilon(0)$.
\end{lem}
\begin{proof}
From \Cref{lem:varphi_H_F} it follows that
 $M_n \in \varphi_A^{-n}(\{M_0\})$.
  Also there is
  $N>0$ such that $\Env_{\{v\}} \subseteq B_{\eps/2}(0)+(\Env_{\{v\}} \cap
  \frac{1}{n}\Z^d)$ for every $n >N$ and
  every $v \in \Ext(A)$.
  By \Cref{lem:Env_cap_Zd_contained_H_F} this implies that for every $v \in \Ext(A)$  there
  exists $u_v \in \Z^d$  such that
  $\Env_{\{v\}} \cap \Z^d \subseteq u_v + H_{\{v\}}$, hence 
  $\Env_{\{v\}} \subseteq
  B_{\eps/2}(0)+(\frac{1}{n}u_v+\frac{1}{n}H_{\{v\}})$.
  Finally for large enough $n$ we have
  
  \begin{align*}
     \bigcup_{i=1}^m (\frac{1}{n}w_i - v_i + \frac{1}{n}H_{v_i}) + B_{\eps/2}(0)
    &\supseteq  \bigcup_{i=1}^m ( \frac{1}{n} u_v - v_i + \frac{1}{n}H_{v_i}) \\
    \frac{1}{n} M_n +B_\eps(0) 
    &= \bigcup_{i=1}^m (\frac{1}{n}w_i - v_i + \frac{1}{n}H_{v_i})  +B_{\eps}(0) \\
    &\supseteq  \bigcup_{i=1}^m (\frac{1}{n}u_v - v_i +
      \frac{1}{n}H_{v_i})  +B_{\eps/2}(0) \\
    &\supseteq  \check{A}.
  \end{align*}
  On the other hand $\Env_{\{v\}}\supseteq H_{\{v\}}$ for every $v \in
  \Ext(A)$ and for sufficiently large $n$
  \begin{equation*}
    \frac{1}{n}M_n=\bigcup_{i=1}^m (\frac{1}{n}w_i-v_i+\frac{1}{n}H_{v_i}) \subseteq \check{A}+B_{\eps}(0).\qedhere
  \end{equation*}
\end{proof}

\begin{lem}\label{lem:space-in-star-arms}
  Let $d \geq 2$ and let $A \subseteq \Z^d$ be a finite positively generating set
  and let $v \in \Ext(A)$.
  There is a point $p \in \R^d$ and an $\eps>0$
  such that
  $B_\eps(p) \cap \check{A} = \emptyset$
  and $(B_\eps(p) + H_{\{v\}}) \cap \conv(-A) = \emptyset$.
\end{lem}
\begin{proof}
  Let $F$ be a $d-1$ dimensional face of $\conv(A)$ containing $v$. Set
  \begin{align*}
    p_{\delta}:=\frac{1}{|\Ext(F)|}\sum_{w \in \Ext(F)} -w + \delta \ell_F.
  \end{align*}
  For all $w \in B_{\epsilon}(p_\delta)+H_{\{v\}}$ we have
  $\langle \ell_F,w \rangle \geq \langle \ell_F,-v\rangle -\epsilon + \delta$
  but for $w \in \conv(-A)$ we have
  $\langle \ell_F,w \rangle \leq \langle \ell_F,-v \rangle$.
  Hence for $\epsilon < \delta$ we have
  \[(B_\epsilon(p_\delta) + H_{\{v\}}) \cap \conv(-A) = \emptyset.\]
  
  To show that $B_{\epsilon}(p_\delta) \cap \check{A} =\emptyset$ for sufficiently
  small $\delta$,
  we have to show that $p_\delta$ is not contained in
  $-\tilde{v}+\Env_{\{\tilde{v}\}}$ for every point $\tilde{v} \in \Ext(A)$. 
  In other words, we have to show that $p_\delta +\tilde{v} \not\in \Env_{\{\tilde{v}\}}$.
  Let $\tilde{F} \neq F$ be a $d-1$ dimensional face of $A$ containing $\tilde{v}$.
  Then all points $u \in \Env_{\{\tilde{v}\}}$ have
  $\langle \ell_{\tilde{F}},u \rangle \geq 0$ but
  \begin{align*}
    \langle \ell_{\tilde{F}}, p_\delta+\tilde{v} \rangle
    = \frac{1}{|\Ext(F)|}
    \sum_{w \in \Ext(F)} \langle \ell_{\tilde{F}},\tilde{v}-w \rangle
    +
    \delta \langle  \ell_{\tilde{F}},  \ell_{F} \rangle
  \end{align*}
  which is negative for small $\delta$ because
  $\langle \ell_{\tilde{F}},\tilde{v}-w \rangle$ is non-positive for all $w
  \in \Ext(F)$ and negative for
  at least one $w \in \Ext(F)$.
\end{proof}

\begin{prop}\label{prop:natural_extension_perfect}
  For any finite positively generating set $A \subseteq \Z^d$, $d\geq 2$, the
  natural extension of $\varphi_A$ is perfect.
\end{prop}
\begin{proof}
  By \Cref{lem:dense-set-in-evtl}
  it suffices to show that for any $M \in \tilde{X}_A$ and $R>0$ there exists $n
  \in \N, W_1,W_2 \in \varphi_A^{-n}\left(\{M\}\right)$ and $W_2 \ne
  W_1$, both in the eventual image of
  $\varphi_A$, such that $M \cap B_R(0)=\varphi_A^{n}(W_1) \cap B_R(0) =
  \varphi_A^{n}(W_2) \cap B_R(0)$.

  Let $M = \bigcup_{i=1}^m(w_i +
  H_{v_i})$ with $ w_1,\ldots,w_n \in \Z^d$ and $\{v_1,\ldots,v_m\} =
  \Ext(A)$.  If we define, as in \Cref{lem:hausdorff-limit-of-preimages},
  \[M_n := \bigcup_{i=1}^m(w_i - n v_i + H_{\{v_i\}}),\] for any $n \in
  \N$, then $M_n \in
  \varphi_A^{-n}(\{M\})$.
  
  By \Cref{lem:space-in-star-arms} there is $\epsilon>0$, a point $p \in \R^d$ and a point $v
  \in \Ext(A)$
  such that $B_\eps(p) \cap \check{A} = \emptyset$
  and $(B_\eps(p)+\Env_{\{v\}}) \cap \conv(-A) = \emptyset$.

  Since $\frac{1}{n}M_n$ converges to $\check{A}$ by
  \Cref{lem:hausdorff-limit-of-preimages}, we can choose $n$ large
  enough such that
  \begin{align*}
    \frac{1}{n} M_n \cap B_{\eps/2}(p)  &= \emptyset, \\
    \frac{1}{n} \Z^d \cap B_{\eps/2}(p)   &\neq \emptyset, \\
    \frac{R}{n} &< \frac{\eps}{2}.
  \end{align*}
  Let $w \in B_{n\eps/2}(p) \cap \Z^d$
  and set $W_1:=M_n$, $W_2:=M_n \cup (w+H_{\{v\}})$.
  Since $W_1 \cap B_{n\eps/2}(p) =
  \emptyset$ and $w \in W_2 \cap B_{n\eps/2}(p)$
  we have $W_2 \neq W_1$.
  We also have $\varphi_A^n(W_2)=M \cup (w+nv+H_{\{v\}})$.

  It remains to show that $(w+nv+H_{\{v\}}) \cap B_R(0) = \emptyset$.
  We know that $(B_{\eps/2}(p) + \Env_{\{v\}})\cap
  (B_{\eps/2}(0)+\conv(-A))=\emptyset$.
  Hence
  \begin{align*}
    (B_{n\eps/2}(np) + \Env_{\{v\}}) \cap (n\conv(-A) + B_{n\eps/2}(0))&=\emptyset,\\
    (w + H_{\{v\}}) \cap (-nv  +B_R(0)) &= \emptyset,\\
    (w+nv+H_{\{v\}}) \cap B_R(0)&=\emptyset.                       \qedhere
  \end{align*}
\end{proof}

Combining \Cref{prop:perfect_NA_implies_unique} and \Cref{prop:natural_extension_perfect} we conclude:
\begin{cor}\label{thm:Zd_natural_extension_is_north_south_pole}
  For any $d >1$ and any finite positively generating set $A \subset
  \Z^d$ that contains $0$, the natural extension of
  $\varphi_A$ is topologically conjugate to the
  (unique) north-south system on the Cantor set.
\end{cor}

\section{The eventual image of $\varphi_A$ ,  $A \subseteq \Z^d$}\label{sec:eventual-image-zd}

In this section we still consider the dynamics of $\varphi_A:\PP(\Z^d) \to \PP(\Z^d)$ where $A$ is a positive generating set of $\Z^d$ containing $0$. Our goal is to  show that the Cantor-Bendixon rank of the eventual image is a non-trivial invariant.
%Throughout this section let $A$ be a positive generating set of $\Z^d$.
More precisely, we show by examples that for $A \subseteq \Z^2$, the
Cantor-Bendixon rank of $\Evt(\varphi_A)$ can be $0,1$ or
$\omega_0$. We suspect that these are the only possibilities, at
least for $d=2$.  Recall that we call elements of $\Evt(\varphi_A)$ \emph{horoballunions}.

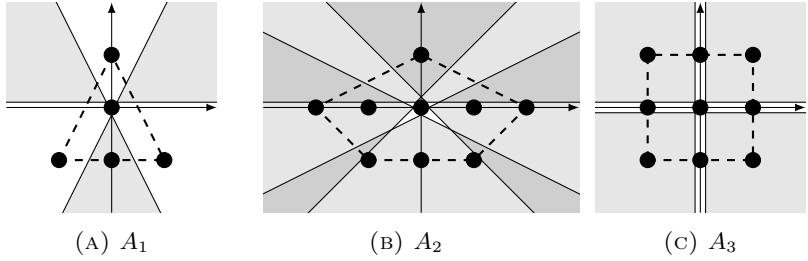
\begin{figure}[h]
  \centering
    \begin{subfigure}{0.3\textwidth}
      \centering
\begin{tikzpicture}[scale=0.7]
\begin{scope}
\clip (-2,-2) rectangle (2,2);
\begin{scope}[every node/.style= {circle,inner sep=0cm, draw=black, fill= black, minimum size=0.2cm}]
\node (A1) at (-1,-1) {};
\node (A2) at (0,-1) {};
\node (A3) at (1,-1) {};
\node (A4) at (0,0) {};
\node (A5) at (0,1) {};
\end{scope}
\draw[thick, dashed] (A1.center) -- (A3.center) -- (A5.center) -- cycle;
\draw[-latex] (0,-2) -- (0,2);
\draw[-latex] (-2,0) -- (2,0);
\draw[fill=black, fill opacity=0.1,xshift=0,yshift=-0.141cm] (-4,-8) -- (0,0) -- (4,-8) -- cycle;
\draw[fill=black, fill opacity=0.1,xshift=0.1cm,yshift=0.1cm] (8,0) -- (0,0) -- (4,8) -- cycle;
\draw[fill=black, fill opacity=0.1,xshift=-0.1cm,yshift=0.1cm]  (-8,0) -- (0,0) -- (-4,8) -- cycle;
\end{scope}
\end{tikzpicture}
\caption{$A_1$}
\end{subfigure}
%%%%%
\begin{subfigure}{0.3\textwidth}
  \centering
\begin{tikzpicture}[scale=0.7]
\begin{scope}
\clip (-3,-2) rectangle (3,2);
\begin{scope}[every node/.style= {circle,inner sep=0cm, draw=black, fill= black, minimum size=0.2cm}]
\node (B11) at (-1,-1) {};
\node (B12) at (0,-1) {};
\node (B13) at (1,-1) {};
\node (B21) at (-2,0) {};
\node (B22) at (-1,0) {};
\node (B23) at (0,0) {};
\node (B24) at (1,0) {};
\node (B25) at (2,0) {};
\node (B31) at (0,1) {};
\end{scope}
\draw[thick, dashed] (B11.center) -- (B13.center) -- (B25.center) -- (B31.center) -- (B21.center) -- cycle;
\draw[-latex] (0,-2) -- (0,2);
\draw[-latex] (-3,0) -- (3,0);
\draw[fill=black, fill opacity=0.1,xshift=0,yshift=-0.141cm] (-8,-4) -- (0,0) -- (8,-4) -- cycle;
\draw[fill=black, fill opacity=0.1,xshift=0.141cm,yshift=0] (8,-8) -- (0,0) -- (8,4) -- cycle;
\draw[fill=black, fill opacity=0.1,xshift=0.1cm,yshift=0.1cm] (8,0) -- (0,0) -- (-40,40) -- cycle;
\draw[fill=black, fill opacity=0.1,xshift=-0.1cm,yshift=0.1cm] (-8,0) -- (0,0) -- (40,40) -- cycle;
\draw[fill=black, fill opacity=0.1,xshift=-0.141cm,yshift=0] (-8,-8) -- (0,0) -- (-8,4) -- cycle;
\end{scope}
\end{tikzpicture}
\caption{$A_2$}
\end{subfigure}
%%%% 
\begin{subfigure}{0.3\textwidth}
  \centering
\begin{tikzpicture}[scale=0.7]
\begin{scope}
\clip (-2,-2) rectangle (2,2);
\begin{scope}[every node/.style= {circle,inner sep=0cm, draw=black, fill= black, minimum size=0.2cm}]
\node (C11) at (-1,-1) {};
\node (C12) at (0,-1) {};
\node (C13) at (1,-1) {};
\node (C21) at (-1,0) {};
\node (C22) at (0,0) {};
\node (C23) at (1,0) {};
\node (C31) at (-1,1) {};
\node (C32) at (0,1) {};
\node (C33) at (1,1) {};
\end{scope}
\draw[thick, dashed] (C11.center) -- (C13.center) -- (C33.center) -- (C31.center) -- cycle;
\draw[-latex] (0,-2) -- (0,2);
\draw[-latex] (-2,0) -- (2,0);
\draw[fill=black, fill opacity=0.1,xshift=-0.1cm ,yshift=-0.1cm] (-8,0) -- (0,0) -- (0,-8) -- cycle;
\draw[fill=black, fill opacity=0.1,xshift=-0.1cm,yshift=0.1cm] (-8,0) -- (0,0) -- (0,8) -- cycle;
\draw[fill=black, fill opacity=0.1,xshift=0.1cm,yshift=-0.1cm] (8,0) -- (0,0) -- (0,-8) -- cycle;
\draw[fill=black, fill opacity=0.1,xshift=0.1cm,yshift=0.1cm] (8,0) -- (0,0) -- (0,8) -- cycle;
\end{scope}
\end{tikzpicture}
\caption{$A_3$}
\end{subfigure}
\caption{Three generating sets whose eventual images have different
  Cantor-Bendixson rank, together with the envelopes corresponding to
  the extremal points of the convex hull, moved slightly outward for greater clarity}
  \label{fig:generators-cb-rank}
\end{figure}

\begin{exam}
  Consider the positive generating sets $A_1,A_2$ and $A_3$ of $\Z^2$ depicted
  in \Cref{fig:generators-cb-rank}. The eventual image of
  $\varphi_{A_i}$ has
  the following structure.
  \begin{enumerate}
  \item $\Evt(\varphi_{A_1})$ is perfect by \Cref{prop:conv_triangle_perfect} below.
   \item $\Evt(\varphi_{A_2})$ has Cantor-Bendixson rank 1 by \Cref{prop:evt_rank_1} below.
  \item $\Evt(\varphi_{A_3})$ has Cantor-Bendixson rank $\omega_0$ by \Cref{prop:rect_rank_omega_0} below.
  \end{enumerate}
\end{exam}

We now turn to prove the statement claimed regarding the example above.

Here is  some ad-hoc terminology:
\begin{defn}
  Let $W \subseteq \Z^d$ be finite and $M \in \Evt(\varphi_A)$.
  We call $\tilde{M} \in \Evt(\varphi_A)$ a \emph{$W$-approximation} of $M$
  if $\tilde{M} \cap W = M \cap W$.
  %We call $M$ \emph{locally minimal} with respect to $W$ if
  %all $W$-approximations of $M$ contain $M$.
\end{defn}

\begin{defn}
  A horoballunion $M \in \Evt(\varphi_A)$ is called \emph{deficient}, if there
  is a finite set of horoballs $u_i+H_{\{v_i\}}, i \in I$ such
  that $M = \bigcup_{i \in I} (u_i+H_{\{v_i\}})$
  and $\bigcup_{i \in I} \Env_{\{v_i\}} \neq \R^d$.
  Denote the set of deficient horoballunions by $D_A$.
\end{defn}

\begin{lem}\label{lem:deficient-perfect}
  The set of all deficient horoballunions has no isolated points.
\end{lem}
\begin{proof}
  Let $M = \bigcup_{i \in I} (u_i+H_{\{v_i\}})$ be a finite deficient
  union of horoballs.  Set
  $U := \{w \in \bigcup_{i \in I} \Env_{\{v_i\}} \setsep ||w||=1\}$.
  Let $u$ be an element in the boundary of $U$ considered as a subset
  of $S^{d-1}=\{w \in \R^d \setsep ||w||=1\}$. We can find
  $v \in \{v_i \setsep i \in I\}$ such that $u \in \Env_{\{v\}}$. Then
  $-u \not \in \Env_{\{v\}}$ and we can find
  $\tilde{u} \in \Z^d \setminus \bigcup_{i \in I} \Env_{\{v_i\}}$ such
  that still $-\tilde{u} \not \in \Env_{\{v\}}$.  Hence, for every finite
  $W \subseteq \Z^d$ we can find $t_0 >0$ such that for all $t>t_0$ we
  have $(t\tilde{u}+\Env_{\{v\}}) \cap W = \emptyset$.  Since
  $M = \bigcup_{i \in I} (u_i+H_{\{v_i\}})$, by our choice of
  $\tilde{u}$ we can find $t_1\in \N$, $t_1>t_0$ such that
  $t_1 \tilde{u} \not\in M$, hence
  $M \cup (t_1 \tilde{u} + H_{\{v\}})$ is a $W$-approximation of $M$
  which is different from $M$ but still deficient.
\end{proof}

The following results will show that there are cases where the Cantor-Bendixon rank is  $0$ or $1$:

\begin{prop}\label{prop:conv_triangle_perfect}
If $A \subseteq \Z^2$ is a positively generating set  with $0 \in A$ and $\conv(A)$ is a triangle then $\Evt(\varphi_A)$ is perfect.	
\end{prop}
\begin{proof}
	Let $A$ be as above.
	Because the sum of the angles of a triangle is $\pi$, which is strictly less than $2\pi$, every finite union
	of $A$-horoballs  is deficient. Hence by \Cref{lem:deficient-perfect}
	$D_{A}$ is a
	dense subset of $\Evt(\varphi_{A})$ without isolated points.
\end{proof}

\begin{prop}\label{prop:evt_rank_1}
	Suppose $A \subseteq \Z^2$ is a positively generating set  with $0 \in A$ that satisfies the following properties:
	\begin{enumerate}
   \item \label{cond:cover} For every
     $w \in \mathbb{R}^2 \setminus \{0\}$ there exists
     $v \in \Ext(A)$ so that $w$ is contained in the interior
     of $\Env_{\{v\}}$.
   \item \label{cond:not_cover}For every $v \in \Ext(A)$ there exists
     $w \in \mathbb{R}^2 \setminus \{0\}$ which is not contained in
      \[\bigcup_{v' \in \Ext(A)\setminus \{v\}}\Env_{\{v'\}}.\]
	\end{enumerate}
Then the Cantor-Bendixon rank of $\Evt(\varphi_A)$ is equal to $1$.
\end{prop}
\begin{proof}
  Let $A$ be as above.  We will prove the result by showing that any
  horoballunion $M$ which is not in the closure of $D_{A}$ is isolated
  in $\Evt(\varphi_A)$. Let $M$ be such a horoballunion.  There exists
  a finite set $W \subseteq \Z^d$ such that all $W$-approximations of
  $M$ are non-deficient. Consider the horoballunions of the form
  $\bigcup_{w \in W} (w+H_{\{v_w\}})$ with
  $\{v_w \setsep w \in W\} = \Ext(A)$.  There are only finitely many
  of them, they are all cofinite by \eqref{cond:cover}, and by
  \eqref{cond:not_cover} every $W$-approximation of $M$ contains one
  of them.  Therefore we can find a finite set $\tilde{W}$
  such that all $\tilde{W}$-approximations of $M$ contain the
  complement of $\tilde{W}$.  Thus $M$ is isolated.  This shows
  that the isolated points of $\Evt(\varphi_{A})$ are precisely all
  points not in $\overline{D_{A}}$.
\end{proof}

The following result demonstrates the case of Cantor-Bendixson rank $\omega_0$.

\begin{prop}\label{prop:rect_rank_omega_0}
	Suppose $n \in \N$ and let $A = \{-1,0,-1\}^2 \subseteq \Z^2$.
	Then $\Evt(\varphi_{A})$ has Cantor-Bendixson rank $\omega_0$.
\end{prop}

\begin{proof}%[Proof of \Cref{prop:rect_rank_omega_0}]
  Let $A = \{-1,0,1\}^2 \subseteq \Z^2$.
    Denote \[B=\{(-1,0),(1,0),(0,-1),(0,1)\} \subseteq \Z^2.\]
    Let  a \emph{ray} be a set of the form $\{(x,y) +tv \setsep t \in
    \N\}$ with $(x,y) \in \Z^2$ and $v \in B$.
    For
    $M \in \Evt(\varphi_A)$ define the \emph{rank} of $M$ as follows.
    If $M$ is in the closure of $D_{A}$, the rank of $M$ is $\infty$.
    If $M$ is not in the closure of $D_{A}$, the rank of $M$ is the
    maximal number of pairwise disjoint rays contained in the
    complement of $M$.

    Let us first establish that for
    $M \not \in \overline{D_A}$, the rank is indeed finite.  Let
    $M \not \in \overline{D_A}$, hence there is a finite set $W$ such
    that all $W$-approximations of $M$ are non-deficient. Hence there
    must be $w_v \in \Z^2$ for $v \in \Ext(A)$ such that all
    $W$-approximations of $M$ contain $\bigcup_{v \in \Ext(A)} (w_v + H_{\{v\}})$. But
    no ray contained in $\{(0,y) + t(1,0) \setsep t \geq 0\}$, with $y$ between
    the y-coordinates of $w_{(-1,-1)}$ and $w_{(-1,1)}$ can be
    contained in the complement of $M$. Hence there can be only
    finitely many pairwise disjoint rays with direction $(1,0)$ in the complement of
    $M$. Similarly for the other directions.

    Now we want to show that
    the set of all horoballunions having rank $k$ or larger is closed.
    Let  $M \not\in
    \overline{D_A}$. We have to show that every sufficiently good
    approximation of $M$ has rank at most that of $M$.
    As above there is a finite set $W$ and there are
    $w_v \in \Z^2$ for $v \in \Ext(A)$ such that all
    $W$-approximations of $M$ contain
    $\tilde{M}:=\bigcup_{v \in \Ext(A)} \left(w_v + H_{\{v\}}\right)$.  Therefore $\tilde{M}$
    has finite rank larger or equal that of $M$.
    Let
    $R=\{(x,y) + t(1,0) \setsep t \in \N\} $ be a ray contained in $M$
    but disjoint from $\tilde{M}$.
    It is now enough to show that every sufficiently good approximation of $M$ also
    contains $R$. Set $L:=\{(x,y) + t(1,0) \setsep t \in \Z\}$
  
  If $L \cap (X \setminus M) \neq
  \emptyset$ then every sufficiently good approximation
  of $M$ by a horoballunion will also contain the ray
  $R$.

  If, on the other hand, $L \subseteq M$ we will show that every
  sufficiently good approximation of $M$ by horoballunions must
  contain $L$.  Assume there are arbitrary good approximations of $M$
  by horoballunions not containing $L$.  It is easy to see that in
  this case either the upper or lower half space defined by $L$ is
  contained in $M$.  Without loss of generality assume it is the lower
  half space, call it $S$, and that $S$ is the maximal one contained
  in $M$.  Every sufficiently good approximation of $M$ must contain
  $\left((x_1,y_1)+H_{(1,1)}\right) \cup \left((x_{-1},y_{-1})+H_{(-1,1)}\right)$
  for points $(x_1,y_1) \in \Z^2$ and $(x_{-1},y_{-1}) \in \Z^2$ with
  $y_1 > y$, $y_{-1}>y$ and $x_{1} < x_{-1}$.  This is due to the fact
  that $M$ is not contained in $\overline{D_A}$ and that $S$ is the
  maximal lower half space contained in $M$.  Consider a vertical line
  $V=\{(a,0) + t(0,-1) \setsep t \in \Z\}$ with $y_{1} < a < y_{-1}$.
  It is now enough to show that $S \cap V$ is contained in every
  sufficiently good approximation of $M$.  We have to treat two cases.
  \begin{enumerate}
    \item The left or the right half space define by $V$ is contained
      in $M$. Assume it is the left one and call it $\tilde{S}$.
      As discussed above, since $M$ is not in the closure of $D_A$,
      there must be a point $\tilde{w} \in \Z^2
      \setminus (S \cup \tilde{S})$ such that $\tilde{w} + H_{\{(1,1)\}}$
    is contained in every sufficiently good approximation of $M$.
    But then $V \cap S \subseteq \tilde{w} + H_{\{(1,1)\}}$
    is also contained in every sufficiently good approximation of $M$.
    A similar argument works if the right half space defined by $V$
    is contained in $M$.
    \item Neither the left nor the right half space defined by $V$
  is contained in $M$.
  Hence $V \cap (\Z^2 \setminus M) \neq \emptyset$, as discussed above, so
  $V \cap S$ is contained in every sufficiently good
  approximation of $M$.
\end{enumerate}
This shows that the whole of $S$
  is contained in every sufficiently good approximation of $M$.
  All in all we thus showed that the horoballunions having rank $k$
  or larger form a closed subset of $\Evt(\varphi_A)$.

  Now we have to show that every horoballunion of rank $k+1$
  can be arbitrarily well approximated by a horoballunion of rank
  $k$. Assume that $\Z^2 \setminus M$ contains
  a ray with direction $(0,1)$ and let $\{(x,y) + t(0,1) \setsep t \in
  \N\}$ be the rightmost of them. Then $(M \cup ((x,n) +
  H_{(-1,-1)}))_{n \in \N}$ is sequence of horoballunions
  converging to $M$ whose rank is $k$ for sufficiently large $n$.
  See \Cref{fig:approx-horoballunion} for an illustration.

  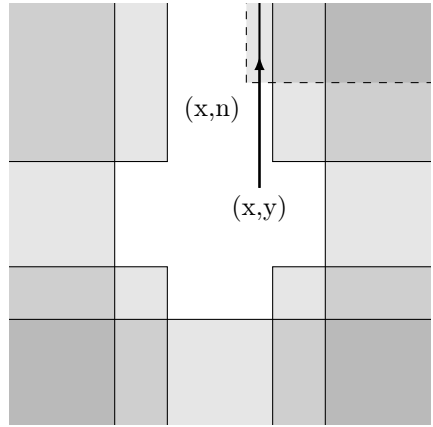
\begin{figure}
    \centering
    \begin{tikzpicture}[scale=0.7]
      \begin{scope}
        \clip (-4,-4) rectangle (4,4);
        \draw[fill=black, fill opacity=0.1] (1,1)  rectangle (10,10) ;
        \draw[fill=black, fill opacity=0.1](-1,-1)  rectangle (-10,-10) ;
        \draw[fill=black, fill opacity=0.1](1,-1)  rectangle (10,-10) ;
        \draw[fill=black, fill opacity=0.1] (-1,1)  rectangle (-10,10) ;
        \draw[fill=black, fill opacity=0.1] (-10,-2)  rectangle (10,-20) ;
        \draw[fill=black, fill opacity=0.1] (-10,-10)  rectangle (-2,10) ;
        \draw[fill=black, fill opacity=0.1] (10,-10)  rectangle (2,10) ;
        
        \draw[fill=black, dashed, fill opacity=0.1] (0.5,2.5)  rectangle (10,10) ;
        \draw[-latex, thick] (0.75,0.5) -- (0.75,10);
        \draw[-latex, thick] (0.75,0.5) -- (0.75,3);
        \node (A) at (0.75,0.1) {(x,y)};
        \node (B) at (-0.15,2) {(x,n)};
      \end{scope}
    \end{tikzpicture}
    \caption{Approximating a horoballunion of rank $k+1$ by one of
      rank $k$}.
    \label{fig:approx-horoballunion}.
  \end{figure}
  
  By induction this shows that $\Evt(\varphi_A)^{(k)}$
  consists of all horoballunions of rank at least $k$
  and the isolated points in this set are those
  of rank precisely $k$. We also saw that $\bigcap_{k \in \N_0}
  (\Evt(\varphi_A))^{(k)} = \overline{D_A}$.
  Finally see \Cref{fig:rank-k-non-empty} for an 
  element in $\Evt(\varphi_A)^{(k)}$, so all these sets are non-empty.
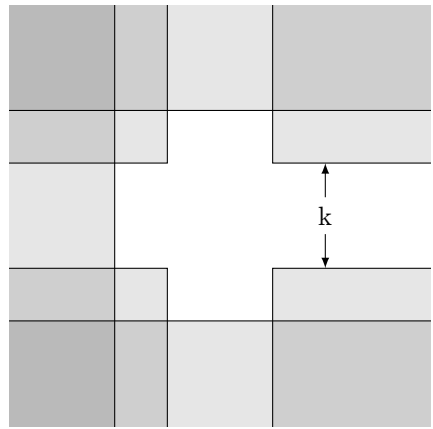
\begin{figure}
  \centering
  \begin{tikzpicture}[scale=0.7]
    \begin{scope}
      \clip (-4,-4) rectangle (4,4);
      \draw[fill=black, fill opacity=0.1](-1,-1)  rectangle (-10,-10) ;
      \draw[fill=black, fill opacity=0.1](1,-1)  rectangle (10,-10) ;
      \draw[fill=black, fill opacity=0.1] (-1,1)  rectangle (-10,10) ;
      \draw[fill=black, fill opacity=0.1] (1,1)  rectangle (10,10) ;
      \draw[fill=black, fill opacity=0.1] (-10,2)  rectangle (10,20) ;
      \draw[fill=black, fill opacity=0.1] (-10,-2)  rectangle (10,-20) ;
      \draw[fill=black, fill opacity=0.1] (-10,-10)  rectangle (-2,10) ;
      \node (A) at (2,0) {k};
      \draw[-latex] (A.north) -- (2,1);
      \draw[-latex] (A.south) -- (2,-1);
    \end{scope}
  \end{tikzpicture}
  \caption{An element of $\Evt(\varphi_A)$ with rank $k$.}
  \label{fig:rank-k-non-empty}
\end{figure}

  All in all we therefore showed that $\Evt(\varphi_A)$ has rank $\omega_0$.
\end{proof}

\section{Factoring and non-factoring results for $\varphi_A$, $A
  \subseteq\Z^d$}\label{sec:factoring}

The main purpose of this final section is to present the following seemingly innocent question, which we have been unable to resolve:

\begin{quest}\label{quest:factoring}
Given two finite positively generating sets $A_1 \subseteq \Z^{d_1}$ and $A_2 \subseteq \Z^{d_2}$ when does $(\PP(\Z^{d_1}),\varphi_{A_1})$ factor onto $(\PP(\Z^{d_2}),\varphi_{A_2})$?
\end{quest}
In fact, apart from the case $d_1=1$ and $d_2>1$, we do not know of any example where the answer is false.  On the other hand, we know very few examples of non-trivial factor maps between such systems.

The rest of  the section presents feeble partial results on the above question.

It is a folklore result that any factor map between  topological dynamical systems induces a factor map between the eventual images and between the natural extensions. The simple proof (below) is  a  compactness argument.
\begin{lem}\label{lem:facor_natural_extension_lift}
  Let $X,Y$ be compact Hausdorff topological spaces and $f: X \to X$ and $g: Y \to Y$ be continuous maps, and let  $\pi: X \to Y$ be a continuous surjective map such that $ \pi \circ f = g \circ \pi$. 
Then:
\begin{enumerate}
\item The restriction of $\pi$ to $\Evt(f)$ is onto $\Evt(g)$.
\item Let  $\hat \pi:\hat X_f \to \hat X_g$ be given by 
  $\hat \pi ((x_n)_{n \in \Z})_k = \pi (x_k)$. Then $\hat{\pi}$ is
  a factor map from $(\hat X_f,\hat f)$ onto $(\hat X_g,\hat g)$.
\end{enumerate} 
\end{lem}
\begin{proof}
  For every $y \in \Evt(g)$ we can find $(y_n)_{n \in \Z} \in \hat X_g$ with
  $y_0=y$ by compactness. It is therefore enough to show that $\hat \pi$ is surjective.
  Let $(y_n)_{n \in \Z} \in \hat X_g$.
  For every $k \in \N$ there is $x^k \in X$ with $y_{-k} = \pi(x^{k})$.
  Set
  \[   x^k_n:= 
    \begin{cases}
      x^k & \text{ if } n < -k, \\
   f^{k+n}(x^k) &\text{ if } n \geq -k .
    \end{cases}
  \]
  Let $x$ be a limit of some subsequence of $((x^k_n)_{n \in \Z})_{k
    \in \N}$. Then $x \in \hat X_f$ and $(y_n)_{n \in \Z} =
    \hat \pi((x_n)_{n \in \Z})$.
  \end{proof}

\begin{prop}
  There is no factor map from $(\PP(\Z),\varphi_{\{-1,0,1\}})$ onto
  $(\PP(\Z^2),\varphi_{\{-1,0,1\}^2})$.
\end{prop}
\begin{proof}
  By \Cref{lem:facor_natural_extension_lift} it suffices to show that
  the natural extension of $(\PP(\Z),\varphi_{\{-1,0,1\}})$ does not
  factor onto the natural extension of
  $(\PP(\Z^2),\varphi_{\{-1,0,1\}^2})$. This follows from the fact
  that the
  natural extension of $(\PP(\Z),\varphi_{\{-1,0,1\}})$ is countable,
  whereas the natural extension of
  $(\PP(\Z^2),\varphi_{\{-1,0,1\}^2})$ is not:

  The eventual image of $(\PP(\Z),\varphi_{\{-1,0,1\}})$ consists of
  elements of the form
  \[\{\Z \setminus (s,t) \setsep s\le t\}\]
  with $s,t \in \Z \cup \{-\infty,\infty\}$.
  In particular it is countable.
  
   On the other hand the eventual image of
   $(\PP(\Z),\varphi_{\{-1,0,1\}})$
   is uncountable, as it contains the pairwise different elements
   $\bigcup_{i \in I} \left((i,-i) +  H_{\{(1,1)\}}\right)$ for $I \subseteq \Z$.
   %Hence there can be no surjective map from
   %$(\PP(\Z),\varphi_{\{-1,0,1\}})$ onto
  %$(\PP(\Z^2),\varphi_{\{-1,0,1\}^2})$.
\end{proof}

We finish with some easy cases where there is a factor map.
\begin{prop}
  There is a factor map from
  \[(\Evt(\varphi_{\{-1,0,1\}^2}),\varphi_{\{-1,0,1\}^2})\] onto
  \[(\Evt(\varphi_{\{-1,0,1\}}),\varphi_{\{-1,0,1\}}).\]
\end{prop}
\begin{proof}
  Define a map $\pi$ from
  $(\Evt(\varphi_{\{-1,0,1\}^2}),\varphi_{\{-1,0,1\}^2})$
  to
  $\PP(\Z)$
  by
\[
\pi(M) := \left\{k \in \Z \setsep ~ (k,k) \in M \mbox{ or } M \not \in X_{|k|} \right\}
\]
  where
  \begin{align*}
    X_k := \{M \in \PP(\Z) \setsep & \exists s,t \in \Z:M \cap
  \{-k,\dots,k\}^2 \\
    &= \left(\left((s,s)+H_{\{(-1,-1)\}}\right) \cup ( (t,t) + H_{\{(1,1)\}})\right)\cap \{-k,\dots,k\}^2\}.
  \end{align*}
Let us check that $\pi$ indeed defines a factor map from   %(\Evt(\varphi_{\{-1,0,1\}^2}),\varphi_{\{-1,0,1\}^2})$ onto
  $(\Evt(\varphi_{\{-1,0,1\}}),\varphi_{\{-1,0,1\}})$:

  The map $\pi$ is continuous, because whether $k \in \pi(M)$
  depends only on $M \cap \{-k,\dots,k\}^2$.  Recall that
\[\Evt(\varphi_{\{-1,0,1\}}) =
  \{\Z \setminus (s,t) \setsep s< t\} \cup
  \{\Z \setminus (s,+\infty) \setsep s \in \Z\} \cup
  \{Z \setminus (-\infty,t) \setsep t \in \Z\}\cup \{\Z,\emptyset \}.
\]
Now for any $s < t$,
\begin{align*}
  \pi
  \left(
    \left((s,s) + H_{\{(-1,-1)\}}\right) \cup
    \left((t,t) + H_{\{(1,1)\}}\right)
  \right)
  &=
  \Z \setminus (s,t).\\
  \pi((s,s) + H_{\{(-1,-1)\}}) &= \Z \setminus (s,+\infty),\\
  \pi((t,t) + H_{\{(1,1)\}}) &= \Z \setminus (-\infty,t),\\
  \pi(\Z^2)&=\Z,~ \pi(\emptyset) = \emptyset.
\end{align*}
Thus the image of $\pi$ contains
  $\Evt(\varphi_{\{-1,0,1\}})$.
  It remains to show that $\pi$ intertwines the dynamics.
  For this we record the following easy facts about the sets $X_k, k \in \N$.
  \begin{enumerate}[(1)]
  \item $X_{k+1} \subseteq X_{k}$, \label{enum:X-k-1}
  \item if $M \in X_{k+1}$, then $\varphi_{\{-1,0,1\}^2}(M) \in
    X_{k}$,  \label{enum:X-k-2}
  % \item if $M \not \in X_{k+1}$ and $\varphi_{\{-1,0,1\}^2}(M) \in
  %   X_k$ then $(\ell,\ell) \in
  %   \varphi_{\{-1,0,1\}^2}(M)$ for all $\ell \in
  %   \{-k,\dots,k\}$. \label{enum:X-k-3}
  \item If $\varphi_{\{-1,0,1\}^2}(M) \in X_k$ and there is $\ell \in
     \{-k,\dots,k\}$ with $(\ell,\ell) \not\in
     \varphi_{\{-1,0,1\}^2}(M)$, then $M \in X_{k+1}$.\label{enum:X-k-3}
  \end{enumerate}
  We first show that  $\varphi_{\{-1,0,1\}}(\pi(M)) \subseteq
  \pi(\varphi_{\{-1,0,1\}^2}(M))$.  Let $k \in
  \varphi_{\{-1,0,1\}}(\pi(M))$, so we must have $\{k-1,k,k+1\} \cap
  \pi(M) \neq
  \emptyset$.  Therefore either
  \[\{(k-1,k-1),(k,k),(k+1,k+1)\} \cap M \neq \emptyset\] or
  \[M \not\in X_{|k-1|} \cap X_{|k|} \cap X_{|k+1|}.\] In the first
  case, $(k,k) \in \varphi_{\{-1,0,1\}^2}(M)$ and $k \in
  \pi(\varphi_{\{-1,0,1\}^2}(M))$.  In the second case by \eqref{enum:X-k-1} $M \not \in
  X_{|k|+1}$ and thus by \eqref{enum:X-k-3} either $\varphi_{\{-1,0,1\}^2}(M) \not \in
  X_{|k|}$ or $(k,k) \in
  \varphi_{\{-1,0,1\}^2}(M)$. In both cases we have $k \in
  \pi(\varphi_{\{-1,0,1\}^2}(M))$.

  Now we show $\varphi_{\{-1,0,1\}}(\pi(M)) \supseteq
  \pi(\varphi_{\{-1,0,1\}^2}(M))$.
  Let $k \in \pi(\varphi_{\{-1,0,1\}^2}(M))$, so
  $(k,k) \in \varphi_{\{-1,0,1\}^2}(M)$
  or $\varphi_{\{-1,0,1\}^2}(M) \not \in X_{|k|}$.
  In the first case \[\{(k-1,k-1),(k,k),(k+1,k+1)\} \cap M \neq
  \emptyset,\] hence $\{k-1,k,k+1\} \cap \pi(M) \neq \emptyset$
  and thus $k \in \varphi_{\{-1,0,1\}}(\pi(M))$.
  In the second case $\varphi_{\{-1,0,1\}^2}(M) \not\in X_{|k|}$ hence
  $M \not\in X_{|k|+1}$. Therefore either $k+1 \in \pi(M)$ or
  $k-1 \in \pi(M)$, depending on the sign of $k$, and thus finally
  $k \in \varphi_{\{-1,0,1\}}(\pi(M))$. 
\end{proof}

Let
$G_\N$ denote the subgraph of $\Cayley(\Z,\{-1,0,1\})$ induced by
$\N$, i.e., $V(G_\N)=\N$ and
\[E(G_\N)=\{ (n,n+1)\setsep n \in \N\} \cup \{ (n+1,n)\setsep n \in
  \N\} \cup \{(n,n) \setsep n \in \N\}.\]
\begin{prop}
  Let $(\Gamma,A)$ be a finitely generated group without dead ends and
  with $A=A^{-1}$.
  Then $(\PP(\Gamma),\varphi_A)$
  factors onto $(\PP(\N), \varphi_{G_{\N}})$.
\end{prop}
\begin{proof}
  The factor map is given by $\pi:\PP(\Gamma) \to \PP(\N)$ with
  \begin{align*}
    0 \in \pi(M) &\iff 1_{\Gamma} \in M  \text{ and } \\
    k \in \pi(M) &\iff (A^k \setminus A^{k-1}) \cap M \neq \emptyset \text{
  for } k >0.\qedhere
  \end{align*}
\end{proof}

\begin{prop}
  Let $A = \{-1,0,1\} \subseteq \Z$.
  There is factor map from $(\PP(\Z),\varphi_{A^n})$ to $(\PP(\Z),\varphi_{A})$
  for all $n \in \N$.
\end{prop}
\begin{proof}
  The factor map is given by $\pi:\PP(\Z) \to \PP(\Z)$ with
  \begin{align*}
   \pi(M)&=\{k \in \Z \setsep \{kn,\dots,kn+(n-1)\} \cap M \neq \emptyset\}.\qedhere
  \end{align*}
\end{proof}

We conclude by stating  the following concrete instance of \Cref{quest:factoring}:
\begin{quest}
 Is there a factor map from $(\PP(\Z),\varphi_{A^3})$ to $(\PP(\Z),\varphi_{A^2})$
 or vice versa?
\end{quest}

\printbibliography
\end{document}